\documentclass[11pt,a4paper,reqno]{amsart} 
\usepackage{amssymb}
\usepackage{latexsym}
\usepackage{amsmath}
\usepackage{mathrsfs}
\usepackage{upgreek}
\usepackage{mathabx}
\usepackage[X2,T1]{fontenc}
\usepackage{cite}
\usepackage{color}

\usepackage{amscd}
\usepackage{tikz-cd}

\usepackage{calc}                   

\newcommand{\scal}[2]{\langle #1,#2\rangle}
\newcommand{\rr}[1]{\mathbf R^{#1}}

\newcommand{\zz}[1]{\mathbf Z^{#1}}

\newcommand{\nm}[2]{\Vert #1\Vert _{#2}}
\newcommand{\NM}[2]{\left \Vert #1\right \Vert _{#2}}
\newcommand{\nmm}[1]{\Vert #1\Vert }

\newcommand{\op}{\operatorname{Op}}

\newcommand{\sets}[2]{\{ \, #1\, ;\, #2\, \} }
\newcommand{\ep}{\varepsilon}
\newcommand{\fy}{\varphi}
\newcommand{\cdo}{\, \cdot \, }
\newcommand{\supp}{\operatorname{supp}}

\newcommand{\ON}{\operatorname{ON}}

\newcommand{\vrum}{\vspace{0.1cm}}

\newcommand{\maclB}{\mathcal B}

\newcommand{\maclD}{\mathcal D}

\newcommand{\maclK}{\mathcal K}
\newcommand{\maclL}{\mathcal L}
\newcommand{\maclM}{\mathcal M}
\newcommand{\maclS}{\mathcal S}
\newcommand{\mascB}{\mathscr B}
\newcommand{\mascE}{\mathscr E}
\newcommand{\mascF}{\mathscr F}

\newcommand{\mascI}{\mathscr I}
\newcommand{\mascP}{\mathscr P}
\newcommand{\mascS}{\mathscr S}

\newcommand{\maclH}{\mathcal H}

\newcommand{\dbar}{d\hspace*{-0.08em}\bar{}\hspace*{0.2em}}

\newcommand{\Sets}[2]{\left \{ \, #1\, ;\, #2\, \right \} }

\numberwithin{equation}{section}          

\newtheorem{thm}{Theorem}
\numberwithin{thm}{section}
\newtheorem*{tom}{\rubrik}
\newcommand{\rubrik}{}
\newtheorem{prop}[thm]{Proposition}

\newtheorem{lemma}[thm]{Lemma}

\theoremstyle{definition}

\newtheorem{defn}[thm]{Definition}
\newtheorem{example}[thm]{Example}

\theoremstyle{remark}
\newtheorem{rem}[thm]{Remark}              

\author{Serap {\"O}ztop}

\address{Department of Mathematics, Faculty of Science,
\.{I}stanbul University,
\.{I}stanbul,
T\"urkiye}

\email{oztops@istanbul.edu.tr}

\author{Joachim Toft}

\address{Department of Mathematics,
Linn{\ae}us University, V{\"a}xj{\"o}, Sweden}

\email{joachim.toft@lnu.se}

\author{R{\"u}ya {\"U}ster}

\address{Department of Mathematics, Faculty of Science,
\.{I}stanbul University,
\.{I}stanbul,
T\"urkiye}

\email{ruya.uster@istanbul.edu.tr}

\title{Fourier integral operators on Orlicz modulation spaces}


\subjclass[2010]{}

\begin{document}

\begin{abstract}
We establish continuity, compactness and
Schatten-von Neumann properties for
Fourier integral operators with
amplitudes in Orlicz modulation
spaces, when acting on other Orlicz
modulation spaces themselves. The phase
functions are non smooth and admit second order 
derivatives in suitable classes of modulation 
spaces.

\end{abstract}

\maketitle

\section{Introduction}\label{sec0} 

\par

The aim of the paper is to investigate
continuity and compactness properties for
Fourier integral operators
with non-smooth amplitudes (or symbols), when 
acting on Orlicz modulation spaces.
In particular we extend continuity 
and compactness properties
of the Fourier integral operators in 
\cite{CT1,CT2,FeGaPr,ToCoGa}, where related results were 
obtained for such operators when acting on 
classical modulation spaces (of Lebesgue types).
These earlier results are in turn extensions
and generalizations of pioneering results
and investigations performed by Boulkhemair in
\cite{Bu1}, where $L^2$-continuity of
subclasses of our Fourier integral operators 
were established.

\par

We recall that Orlicz versions of Lebesgue type 
spaces are obtained by replacing involved Lebesgue 
norms with Orlicz norms, which are parameterized 
with so-called Young functions. By choosing these 
Young functions in suitable ways one recovers the 
definitions of Lebesgue norms. Hence Orlicz type 
spaces extend the notions of analogous Lebesgue type 
spaces. In our situation, the family of Orlicz 
modulation spaces contains the corresponding family 
of classical modulation spaces.

\par

A Fourier integral operator is an operator
$\op _\fy (a)$, which is parameterized by the
amplitude (or symbol) $a$ and phase function $\fy$, defined
on (suitable extensions of) the phase space. 
For appropriate
$a \in \mascS '(\rr {2d+m})$ 
and real-valued
$\fy \in C(\rr {2d+m})$, $\op _\fy (a)$
is an operator from
$\mascS (\rr {d})$ to $\mascS '(\rr {d})$,
which is commonly defined as
\begin{equation}
\label{Eq:FIntOp}
\begin{aligned}
\op _\fy (a)f(x)
&=
(2\pi )^{-\frac 12(d+m)}
\iint _{\rr {d+m}}a(x,y,\zeta
)f(y)e^{i\fy (x,y,\zeta )}\, dyd\zeta ,
\\[1ex]
f&\in \mascS (\rr {d}),\quad x\in \rr {d}.
\end{aligned}
\end{equation}
(See \cite{H} and Section \ref{sec1} for notations.)
To some extent, the amplitude $a$ quantifies the
amplification, while the phase function $\fy$
is linked to (non-homogeneous) modulations within 
the systems.

\par

In various situations were Fourier integral 
operators are used, the right choice of the phase 
function $\fy$ is crucial. In hyperbolic problems,
which is a common field for applying
Fourier integral operators, one usually assumes that
$\fy (x,y,\zeta )$ is smooth when $\zeta \neq 0$ and 
positively homogeneous of order one with respect to the
$\zeta$ variable. In these hyperbolic situations, the adapted
$\fy (x,y,\zeta )$ usually fails to be differentiable in the 
$\zeta$ variable when $\zeta =0$. Furthermore,
one usually impose strong regularity assumptions on
the amplitudes, e.{\,}g. that they should belong to
subclasses of the H{\"o}rmander class $S^r_{0,0}$, 
or that they should belong to so-called SG-classes.
(See e.{\,}g. 
\cite{BoCoPeTo,Cap,Cori1,CoJoTo,CoriRod,
CorRuz,RuS1,RuS4,RuS5}
and the references therein.)

%
%

\par

Roughly speaking, in contrast to the previous 
assumptions adapted to hyperbolic problems,
in our situation we use the framework in
\cite{Bu1} by Boulkhemair and in \cite{ToCoGa}
concerning (lack of) regularity assumptions on
the amplitude $a$ and phase function $\fy$.
Especially we impose less restrictions
on $\fy (x,y,\zeta )$ outside $\zeta =0$, but 
stronger
regularity with respect to $\zeta$ when $\zeta =0$,
compared to the situation of hyperbolic problems.
Furthermore, we stress that in similar ways as in
\cite{Bu1,ToCoGa}, the regularity assumptions on the 
amplitude $a$ are also relaxed compared to what 
is common for Fourier integral operators.


%
%
%

\par

In order to be more specific, suppose that
$\omega$, $\omega _1$, $\omega _2$ and $v$ are 
suitable weight functions, and that $\Phi$
is a suitable Young function. Also suppose that
$a$ belongs to the modulation space
$M^{\infty ,1}_{(\omega )}(\rr {2d+m})$
(of so-called Sj{\"o}strand class), and that
$$
\fy ''\in M^{\infty ,1}_{(v)}(\rr {2d+m})
$$
satisfies the non-degeneracy condition
\begin{equation}
\label{Eq:DetphiCond}
\left |
\det
\left (
\begin{matrix}
\fy ''_{y,x} & & \fy ''_{\zeta ,x}
\\[1ex]
\fy ''_{y,\zeta} &  & \fy ''_{\zeta ,\zeta}
\end{matrix}
\right )
\right |
\ge
\dbar .
\end{equation}
for some $\dbar >0$. Then Theorem \ref{Thm:Cont1}
in Section \ref{sec2} asserts that $\op _\fy (a)$
is continuous between the Orlicz modulation spaces
$M^{\Phi}_{(\omega _1)}(\rr d)$ and
$M^{\Phi}_{(\omega _2)}(\rr d)$. That is,
the map
\begin{equation}
\label{Eq:IntroCont1}
\op _\fy (a) : M^{\Phi}_{(\omega _1)}(\rr d)
\to
M^{\Phi}_{(\omega _2)}(\rr d)
\end{equation}
is continuous.
If, more restrictive,
$a$ is chosen in the completion
$M^{\sharp ,1}_{(\omega )}(\rr {2d+m})$
of
$M^{1,1}_{(\omega )}(\rr {2d+m})$
in $M^{\infty,1}_{(\omega )}
(\rr {2d+m})$,
then the map \eqref{Eq:IntroCont1}
is compact. (See Theorem 
\ref{Thm:Cont1Comp}.)

\par

We notice that we may choose $\omega$
such that $M^{\infty ,1}_{(\omega )}$
contain $S^r_{0,0}$.
In particular, as announced above, 
we put less restrictions on
the amplitudes compared to the assumptions
above adapted to hyperbolic problems.

\par

We remark that our results extend and generalize
related results in \cite{ToCoGa} in especially
two different ways.

\par

Firstly, we may choose the 
involved Young functions such that
$M^\Phi _{(\omega _j)}$
in \eqref{Eq:IntroCont1}
is equal to the classical modulation space
$M^p _{(\omega _j)}$, for any $p\in [1,\infty]$,
$j=1,2$.

\par

Secondly, in \eqref{Eq:IntroCont1} we allow
$\omega$, $\omega _j$ and $v$ to belong to
the large class $\mascP _s$ of weights which are moderate
by subexponential functions of degree $s^{-1}<1$, while
in \cite{ToCoGa} it is required that the involed
weights should belong to the smaller class $\mascP$
of \emph{polynomially} moderate weights.
By imposing these two significant restrictions, then
our Theorem \ref{Thm:Cont1} in Section \ref{sec2}, 
as well as \eqref{Eq:IntroCont1}, essentially
takes the form \cite[Theorem 2.2]{ToCoGa}.

\par

In similar ways, we deduce several other
continuity properties, including detailed compactness
properties, for Fourier integral
operators on Orlicz modulation spaces, which
cover the continuity 
results in \cite{Bu1,CT1,CT2,ToCoGa}.
For example, we investigate 
$\op _\fy (a)$ when the amplitude
$a$ satisfies conditions which are 
rather similar to norm estimates with respect
to the Orlicz modulation space
$M^{\Phi}_{(\omega )}(\rr {2d+m})$.
For suitable non-degeneracy conditions
on $\fy$, different compared to
\eqref{Eq:DetphiCond}, we show that
$$
\op _\fy (a):M^{\Phi ^*}_{(\omega _1)}(\rr d)
\to
M^{\Phi}_{(\omega _2)}(\rr d)
$$
is continuous.
(See Theorem \ref{Thm:Cont2} in
Section \ref{sec2}.)


\par

In Section \ref{sec3} we perform detailed
studies on compactness for Fourier integral
operators. Here we find necessary
conditions on the amplitudes in order for
corresponding Fourier integral operators
should belong to certain Orlicz Schatten-von
Neumann classes. We impose certain restrictions
on the amplitudes. In the first step we
assume that the amplitudes in
\eqref{Eq:FIntOp} are independent of the
$y$ variable. That is, they are of the
form
\begin{equation}
\label{Eq:FIntOp2}
\begin{aligned}
\op _\fy (a)f(x)
&=
(2\pi )^{-d}
\iint _{\rr {2d}}a(x,\zeta
)f(y)e^{i\fy (x,y,\zeta )}\, dyd\zeta ,
\\[1ex]
f&\in \mascS (\rr {d}),\quad x\in \rr {d}.
\end{aligned}
\end{equation}
Thereafter we consider a more general family of
Fourier integral operators, given by
\begin{equation}
\label{Eq:FIntOp3}
\begin{aligned}
\op _{A,\fy} (a)f(x)
&=
(2\pi )^{-d}
\iint _{\rr {2d}}a(x-A(x-y),\zeta
)f(y)e^{i\fy (x,y,\zeta )}\, dyd\zeta ,
\\[1ex]
f&\in \mascS (\rr {d}),\quad x\in \rr {d}.
\end{aligned}
\end{equation}
Again we deduce Orlicz Schatten-von Neumann properties
for such Fourier integral operators. In fact, for amplitudes $a$ in
\eqref{Eq:FIntOp2} (or more general \eqref{Eq:FIntOp3}), in
suitable weighted $M^\Phi$ classes, and suitable phase functions
$\fy$, we show that corresponding Fourier integral operators are
Schatten-von Neumann operators of order $\Phi$ from $M^2_{(\omega _1)}$
to $M^2_{(\omega _2)}$.

\par

Finally we remark that by choosing
$$
m=d
\quad \text{and}\quad
\fy (x,y,\zeta )\equiv \scal {x-y}\zeta ,
$$ 
then our Fourier integral
operator \eqref{Eq:FIntOp}
becomes the \emph{pseudo-differential operator}
\begin{equation}
\label{Eq:AmplPseudos}
\op (a)f(x)
=
(2\pi )^{-d}\iint _{\rr d}
a(x,y,\zeta )f(y)e^{i\scal{x-y}\zeta}
\, dyd\zeta .
\end{equation}
In particular, pseudo-differential operators
are special cases of Fourier integral operators. 
Furthermore, if
$A$ is a fixed real $d\times d$ matrix,
and instead $a\in \mascS '(\rr {2d})$ is an appropriate function or
distribution on $\rr {2d}$ instead of $\rr {3d}$, 
then pseudo-differential operators of the form
\begin{equation}
\label{Eq:APseudos}
\op _A(a)f(x)
=
(2\pi )^{-d}\iint _{\rr {2d}}
a(x-A(x-y),\xi )f(y)e^{i\scal{x-y}\xi}
\, dyd\xi ,
\end{equation}
can be considered as special case of operators in
\eqref{Eq:FIntOp3}, as well as in
\eqref{Eq:AmplPseudos}.

\par

On the other hand, by Fourier inversion formula
and kernel theorems it follows
that any continuous operator from
$\mascS (\rr d)$ to $\mascS '(\rr d)$
is given by \eqref{Eq:APseudos} for
a suitable choice of $a$. Consequently, the
\emph{set of operators} are in general not increased
by passing from the more restricted formulation
\eqref{Eq:APseudos} to the more general
formulations \eqref{Eq:FIntOp} via
\eqref{Eq:FIntOp3} or
\eqref{Eq:AmplPseudos}.

\par

The additional assumptions on the phase function
for the pseudo-differential operators
in \eqref{Eq:AmplPseudos} and \eqref{Eq:APseudos}
lead to more general continuity properties, compared
to what is possible for general Fourier integral
operators in \eqref{Eq:FIntOp}. For related 
continuity properties for pseudo-differential 
operators when acting on Orclicz modulation spaces,
see \cite{GuRaToUs,TofUst}. For some further
extensions to more general modulation spaces,
see \cite{GroRez,ToPfTe} and the references therein.
For earlier approaches restricted to classical
modulation spaces, see
e.{\,}g. \cite{CorNic2,CorRod,To7,To8} and the
references therein.

\par

\par

\section{Preliminaries}\label{sec1} 

\par

In the section we recall some basic facts on Gelfand-Shilov
spaces, Orlicz spaces, Orlicz modulation spaces,
pseudo-differential operators and Wigner distributions.
We also give some examples on Young functions,
Orlicz spaces and Orlicz modulation spaces. 
(See Examples \ref{Example:SpecYoungFunc} and
\ref{Ex:OrlModSpace1}.) Notice that
Young functions are fundamental in the definition of Orlicz
spaces and Orlicz modulation spaces).

\par

\subsection{Gelfand-Shilov spaces}\label{subsec-Gelfand-Shilov}

\par

For a real number $s>0$, the (standard Fourier invariant) Gelfand-Shilov
space $\mathcal S_{s}(\rr d)$
($\Sigma _{s}(\rr d)$) of Roumieu type (Beurling type)
consists of all $f\in C^\infty (\rr d)$ such that
\begin{equation}\label{gfseminorm}
\nm f{\mathcal S_{s,h}}
\equiv
\sup_{\substack{\alpha, \beta \in \mathbf N^d \\ x\in \rr d}}
\frac {|x^\beta \partial ^\alpha
f(x)|}{h^{|\alpha  + \beta |}(\alpha ! \beta !)^s}
\end{equation}
is finite for some $h>0$ (for every $h>0$). We equip
$\mathcal S_{s}(\rr d)$ ($\Sigma _{s}(\rr d)$) by the canonical inductive limit
topology (projective limit topology) with respect to $h>0$, induced by
the semi-norms defined in \eqref{gfseminorm}.

\par

We have
\begin{equation}\label{GSembeddings}
\begin{aligned}
\maclS _s (\rr d) &\hookrightarrow \Sigma _t(\rr d)
\hookrightarrow \maclS _t (\rr d)
\hookrightarrow
\mascS (\rr d)
\\[1ex]
&\hookrightarrow \mascS '(\rr d) 
\hookrightarrow  \maclS _t'(\rr d)
\hookrightarrow  \Sigma _t'(\rr d) \hookrightarrow
\maclS _s '(\rr d),
\quad \frac 12\le s<t,
\end{aligned}
\end{equation}
with dense embeddings.
Here $A\hookrightarrow B$ means that
the topological space $A$ is continuously embedded in the 
topological space $B$. We also have
$$
\maclS _s(\rr d)=\Sigma _t(\rr d)= \{ 0\} ,\qquad
s<\frac 12,\ t\le \frac 12.
$$

\par

The \emph{Gelfand-Shilov distribution spaces} $\maclS _s'(\rr d)$ 
and $\Sigma _s'(\rr d)$, of Roumieu and Beurling 
types respectively, are the (strong) duals of $\mathcal S_s(\rr d)$ 
and $\Sigma _s(\rr d)$, respectively. It follows that if
$\mathcal S_{s,h}'(\rr d)$ is the $L^2$-dual of
$\mathcal S_{s,h}(\rr d)$ and $s\ge \frac 12$
($s > \frac 12$),
then $\mathcal S_s'(\rr d)$
($\Sigma _s'(\rr d)$) can be identified
with the projective limit (inductive limit) of
$\mathcal S_{s,h}'(\rr d)$ with respect to $h>0$. It follows that
\begin{equation}\label{Eq:GSspacecond2}
\mathcal S_s'(\rr d) = \bigcap _{h>0}\mathcal S_{s,h}'(\rr d)
\quad \text{and}\quad \Sigma _s'(\rr d) =\bigcup _{h>0}
\mathcal S_{s,h}'(\rr d)
\end{equation}
for such choices of $s$ and $\sigma$, see
\cite{GS,Pil1,Pil3} for details.


\par

We let the Fourier transform
$\mathscr F$ be given by
$$
(\mathscr Ff)(\xi )= \widehat f(\xi )
\equiv
(2\pi )^{-\frac d2}\int _{\rr
{d}} f(x)e^{-i\scal  x\xi }\, dx,
\quad \xi \in \rr d,
$$
when $f\in L^1(\rr d)$. Here $\scal \cdo \cdo$
denotes the usual
scalar product on $\rr d$.
The Fourier transform $\mathscr F$ extends
uniquely to homeomorphisms on $\mathscr S'(\rr d)$,
$\maclS _s'(\rr d)$ and on $\Sigma _s'(\rr d)$.
Furthermore,
$\mascF$ restricts to
homeomorphisms on $\mathscr S(\rr d)$,
$\maclS _s(\rr d)$ and on $\Sigma _s (\rr d)$,
and to a unitary operator on $L^2(\rr d)$.
Similar facts hold true
with partial Fourier transforms in place of
Fourier transform.

\par

Let $\phi \in \mascS  (\rr d)$ be fixed.
Then the \emph{short-time
Fourier transform} $V_\phi f$ of $f\in \mascS '
(\rr d)$ with respect to the \emph{window function} $\phi$ is
the tempered distribution on $\rr {2d}$, defined by
\begin{align}
V_\phi f(x,\xi )
&=
\mascF (f \, \overline {\phi (\cdo -x)})(\xi ), \quad x,\xi \in \rr d.
\label{Eq:STFTDef}
\intertext{In some situations it is convenient to use
the small modification}
T_\phi f(x,\xi )
&=
\mascF (f((\cdo +x)) \, \overline {\phi})(\xi ), \quad x,\xi \in \rr d
\label{Eq:TTransfDef}
\end{align}
of $V_\phi f$. By a straight-forward change of
variables it follows that
$$
T_\phi f(x,\xi ) = e^{i\scal x\xi}V_\phi f(x,\xi )
$$
If $f ,\phi \in \mascS (\rr d)$, then it follows that
$$
V_\phi f(x,\xi ) = (2\pi )^{-\frac d2}\int _{\rr d} f(y)\overline {\phi
(y-x)}e^{-i\scal y\xi}\, dy, \quad x,\xi \in \rr d.
$$

\par

By \cite[Theorem 2.3]{Toft28} it follows that the definition of the map
$(f,\phi)\mapsto V_{\phi} f$ from $\mascS (\rr d) \times \mascS (\rr d)$
to $\mascS(\rr {2d})$ is uniquely extendable to a continuous map from
$\maclS _s'(\rr d)\times \maclS_s'(\rr d)$
to $\maclS_s'(\rr {2d})$, and restricts to a continuous map
from $\maclS _s (\rr d)\times \maclS _s (\rr d)$
to $\maclS _s(\rr {2d})$.
The same conclusion holds with $\Sigma _s$ in place of
$\maclS_s$, at each occurrence.

\par

In the following proposition we give characterizations of 
Gelfand-Shilov
spaces and their distribution spaces in terms of estimates
of the short-time Fourier transform.
We omit the proof since the first part follows from
\cite[Theorem 2.7]{GroZim}
and the second part from \cite[Theorem 2.5]{Toft28}.
See also \cite{CoPiRoTe10} for related results.
Here and in what follows, the notation
$A(\theta )\lesssim B(\theta )$, $\theta \in \Omega$,
means that there is a constant $c>0$ such that
$A(\theta )\le cB(\theta )$ holds
for all $\theta \in \Omega$. We also set
$A(\theta )\asymp B(\theta )$ when
$A(\theta )\lesssim B(\theta )\lesssim A(\theta )$.

\par

\begin{prop}\label{stftGelfand2}
Let $s\ge \frac 12$ ($s>\frac 12$), $\phi \in \maclS _s(\rr d)\setminus 0$
($\phi \in \Sigma _s(\rr d)\setminus 0$) and let $f$ be a
Gelfand-Shilov distribution on $\rr d$. Then the following is true:
\begin{enumerate}
\item $f\in \maclS _s (\rr d)$ ($f\in \Sigma_s(\rr d)$), if and only if
\begin{equation}\label{stftexpest2}
|V_\phi f(x,\xi )| \lesssim  e^{-r (|x|^{\frac 1s}+|\xi |^{\frac 1s})}, \quad x,\xi \in \rr d,
\end{equation}
for some $r > 0$ (for every $r>0$).
\item $f\in \maclS _s'(\rr d)$ ($f\in \Sigma _s'(\rr d)$), if and only if
\begin{equation}\label{stftexpest2Dist}
|V_\phi f(x,\xi )| \lesssim  e^{r(|x|^{\frac 1s}+|\xi |^{\frac 1s})}, \quad
x,\xi \in \rr d,
\end{equation}
for every $r > 0$ (for some $r > 0$).
\end{enumerate}
\end{prop}

\par

In our investigations, also
compactly supported elements in
Gelfand-Shilov spaces appears. For
this reason we set
$$
\maclD _s (K)
\equiv
\maclS _s(\rr d)\bigcap \mascE '(K)
\quad \text{and}\quad
\maclD _{0,s} (K)
\equiv
\Sigma _s(\rr d)\bigcap \mascE '(K),
$$
when $K\subseteq \rr d$ is compact,
with topologies induced by the topologies
from $\maclS _s(\rr d)$ and $\Sigma _s(\rr d)$,
respectively. We also let
\begin{equation}
\label{Eq:CompSuppGS}
\maclD _s (\rr d)
=
\bigcup _{j=1}^\infty \maclD _s (K_j)
\quad \text{and}\quad
\maclD _{0,s} (\rr d)=\bigcup _{j=1}^\infty
\maclD _{0,s} (K_j),
\end{equation}
where 
$$
K_j=\sets {x\in \rr d}{|x|\le j}
$$
is the closed ball of radius $j$ with center at 
origin. We let the topologies of
$\maclD _s (\rr d)$ and
$\maclD _{0,s} (\rr d)$ be the inductive limit
topologies of $\maclD _s (K_j)$ and
$\maclD _{0,s} (K_j)$ with respect to $j$.

\par

It is well-known that if $s>1$, then
$$
\maclD _{0,s} (\rr d) \subseteq
\maclD _s (\rr d) \subseteq C_0^\infty (\rr d)
\subseteq \mascS (\rr d),
$$
with dense embeddings. (See e.{\,}g. Sections 1.3
and 8.4 in \cite{H}.) On the other
hand, if $s\le 1$, then
$$
\maclD _{0,s} (\rr d)
=
\maclD _s (\rr d)
=
\{ 0\} ,
$$
that is, the spaces in \eqref{Eq:CompSuppGS}
becomes trivial. This follows from the fact 
that for $s\le 1$, then all elements in
spaces in \eqref{Eq:CompSuppGS} are real
analytic functions.

\par

\subsection{Weight functions}\label{subsec1.2}

\par

A \emph{weight} or \emph{weight function} on $\rr d$ is a
positive function $\omega
\in  L^\infty _{loc}(\rr d)$ such that $1/\omega \in  L^\infty _{loc}(\rr d)$.
The weight $\omega$ is called \emph{moderate},
if there is a positive weight $v$ on $\rr d$ such that
\begin{equation}\label{moderate}
\omega (x+y) \lesssim \omega (x)v(y),\qquad x,y\in \rr d.
\end{equation}
If $\omega$ and $v$ are weights on $\rr d$ such that
\eqref{moderate} holds, then $\omega$ is also called
\emph{$v$-moderate}.
We note that \eqref{moderate}
implies that $\omega$ fulfills
the estimates
\begin{equation}\label{moderateconseq}
v(-x)^{-1}\lesssim \omega (x)\lesssim v(x),\quad x\in \rr d.
\end{equation}
We let $\mascP _E(\rr d)$ be the set of all moderate weights on $\rr d$.

\par

It can be proved that if $\omega \in \mascP _E(\rr d)$, then
$\omega$ is $v$-moderate for some $v(x) = e^{r|x|}$, provided the
positive constant $r$ is large enough (cf. \cite{Gro2.5}).
That is,
\eqref{moderate} implies
\begin{equation}\label{Eq:weight0}
\omega (x+y) \lesssim \omega(x) e^{r|y|}
\end{equation}
for some $r>0$. In particular, \eqref{moderateconseq} shows that
for any $\omega \in \mascP_E(\rr d)$, there is a constant $r>0$ such that
\begin{equation}\label{Eq:BoundWeights}
e^{-r|x|}\lesssim \omega (x)\lesssim e^{r|x|},\quad x\in \rr d.
\end{equation}

\par

We say that $v$ is
\emph{submultiplicative} if $v$ is even and
\eqref{moderate}
holds with $\omega =v$. In the sequel, $v$ and $v_j$ for
$j\ge 0$, always stand for submultiplicative weights if
nothing else is stated.

\par

For any $s>0$, we let $\mascP _s(\rr d)$ be the set 
of all weights $\omega$ on $\rr d$ such that 
\begin{equation}
\label{Eq:weightCondS}
\omega (x+y) \lesssim \omega(x) e^{r|y|^{\frac 1s}}
\end{equation}
holds for some $r>0$. In the same manner the
set $\mascP _{0,s}(\rr d)$ consists of all 
weights $\omega$ on $\rr d$ such that
\eqref{Eq:weightCondS} is true for \emph{every}
$r>0$.
We also let $\mascP (\rr d)$
be the set of all $\omega \in \mascP _E(\rr d)$
such that
$$
\omega (x+y) \lesssim \omega(x) (1+|y|)^r
$$
for some $r>0$.

\par

Evidently,
$$
\mascP (\rr d) \subseteq \mascP _{s_1}(\rr d) 
\subseteq \mascP _{0,s_2}(\rr d)
\subseteq \mascP _{s_2}(\rr d),
\qquad 0<s_1<s_2.
$$
On the other hand, in view of
\eqref{Eq:weightCondS} it follows that
all weights in
$\mascP _{0,s}(\rr d)$ and $\mascP _{s}(\rr d)$
are moderate. Hence \eqref{Eq:weight0} gives
$$
\mascP _{s_1}(\rr d) 
=
\mascP _{0,s_2}(\rr d)
=
\mascP _E(\rr d)
\quad \text{when}\quad
s_1\ge 1,\ s_2>1.
$$

\par

\subsection{Orlicz Spaces}\label{subsec1.3}

\par

We recall that a function $\Phi:[0,\infty ] \to
[0,\infty ]$ is called \emph{convex} if
\begin{equation*}
\Phi(s_1 t_1+ s_2 t_2)
\leq s_1 \Phi(t_1)+s_2\Phi(t_2),
\end{equation*}
when
$s_j,t_j\in \mathbf{R}$
satisfy $s_j,t_j \ge 0$ and
$s_1 + s_2 = 1,\ j=1,2$.

\par

\begin{defn}\label{Def:YoungFunc}
A function $\Phi _0$ from $[0,\infty ]$ to
$[0,\infty ]$
is called a \emph{Young function} if
the following is true:
\begin{enumerate}
\item $\Phi _0$ is convex;

\vrum

\item $\Phi _0(0)=0$;

\vrum

\item $\lim
\limits _{t\to\infty} \Phi _0(t)=\Phi _0(\infty )=\infty$.
\end{enumerate}
\end{defn}

\par

We observe that $\Phi _0$ and $\Phi$
in Definition \ref{Def:YoungFunc} might not
be continuous, because we permit
$\infty$ as function value. For example,
$$
\Phi (t)=
\begin{cases}
0,&\text{when}\ t \leq a
\\[1ex]
\infty ,&\text{when}\ t>a
\end{cases}
$$
is convex but discontinuous at $t=a$.

\par

It is clear that $\Phi _0$ and $\Phi$ in
Definition \ref{Def:YoungFunc} are
non-decreasing, because if $0\leq t_1\leq t_2$
and $s\in [0,1]$ is chosen such
that $t_1=st_2$ and $\Phi _0$ is the same as in
Definition \ref{Def:YoungFunc}, then
\begin{equation*}
    \Phi _0(t_1)=\Phi _0(st_2+(1-s)0)
    \leq s\Phi _0(t_2)+(1-s)\Phi _0(0)
    \leq \Phi _0(t_2),
\end{equation*}
since $\Phi _0(0) =0$ and $s\in [0,1]$. Hence every
Young function is increasing.

\par

\begin{defn}\label{Def:OrliczSpaces1}
Let $\Phi$ be a 
Young function and
let $\omega _0 \in \mascP _E(\rr d)$.
Then the Orlicz space
$L^{\Phi}_{(\omega_0)}(\rr d)$ consists
of all measurable functions
$f:\rr d \to \mathbf C$ such that
$$
\nm f{L^{\Phi}_{(\omega_0)}}
\equiv
\inf  \Sets{\lambda>0}
{\int_\Omega \Phi 
\left (
\frac{|f(x) \cdot \omega_0 (x)|}{\lambda}
\right )
\, dx\leq 1}
$$
is finite. Here $f$ and $g$ in $L^{\Phi}_{(\omega_0)}(\rr d)$
are equivalent if $f=g$ a.e.
\end{defn}

\par






\par

In most of our situations we assume that $\Phi$
and $\Phi _j$ above are Young functions. A few
properties for Wigner distributions in Section
\ref{sec2} are deduced when $\Phi$ and $\Phi _j$
are allowed to be 
Young functions. The reader
who is not interested of such general results may
always assume that all 
Young functions
should be Young functions.

\par

It is well-known that if $\Phi$ 
in Definition \ref{Def:OrliczSpaces1}
is a Young function, then the space
$L^{\Phi}_{(\omega _0)}(\rr d)$ and
$L^{\Phi _1,\Phi _2}_{(\omega )}(\rr {2d})$
is a Banach spaces (see e.{\,}g.
Theorem 3 of
III.3.2 and Theorem 10 of III.3.3 in \cite{RaoRen}).








We refer to \cite[Lemma 1.18]{TofUst}
for the proof of the following lemma.

\par

\begin{lemma}\label{T}
Let $\Phi , \Phi _j$ be 
Young functions, $j=1,2$,
$\omega_0, v_0 \in \mascP_E (\rr d)$
and $\omega, v \in \mascP_E (\rr dd)$ be such that $\omega_0$
is $v_0$-moderate and $\omega$ is $v$-moderate.
Then $L^{\Phi}_{(\omega_0)}(\rr d)$
are invariant under translations, and
$$
\Vert f(\cdo - x)\Vert_{L^\Phi _{(\omega_0)}}
\lesssim
\Vert f\Vert_{L^\Phi _{(\omega_0)}} v_0(x),
\quad
f\in L^\Phi _{(\omega_0)}(\rr d),\ x\in \rr d.
$$
\end{lemma}

\par

In most situations we assume that the 
Young
functions should satisfy the $\Delta _2$-condition
(near origin), whose definition is recalled
as follows.

\par

\begin{defn}\label{Def:Delta2Cond}
Let $\Phi: [0,\infty ] \to [0,\infty]$ be a 
Young function.
Then $\Phi$ is said to satisfy the \emph{$\Delta_2$-condition}
if there exists a constant $C>0$ such that
\begin{equation}\label{Eq:Delta2Cond}
\Phi(2t) \leq C \Phi(t) 
\end{equation}
for every $t\in [0,\infty ]$. The 
Young function $\Phi$
is said to satisfy \emph{local $\Delta_2$-condition}
or \emph{$\Delta_2$-condition near origin}, if there are
constants $r>0$ and $C>0$ such that \eqref{Eq:Delta2Cond}
holds when $t\in [0,r]$.
\end{defn}

\par

\begin{rem}\label{Rem:Delta2Cond}
Suppose that $\Phi: [0,\infty ] \to [0,\infty]$
is a 
Young function which satisfies \eqref{Eq:Delta2Cond}
when $t\in [0,r]$ for some constants $r>0$ and $C>0$.
Then it follows by straight-forward arguments that
there is a 
Young function $\Phi _0$ (of the same order)
which satisfies the $\Delta _2$-condition
(on the whole $[0,\infty $), and such that
$\Phi _0(t)=\Phi (t)$ when $t\in [0,r]$).
\end{rem}

\par

Several duality properties for Orlicz spaces
can be described in terms of Orlicz spaces
with respect to Young conjugates, given in
the following definition.

\par

\begin{defn}\label{Def:ConjYoungFunc}
Let $\Phi$ be a Young function. Then
the conjugate Young function
$\Phi ^*$ is given
by
\begin{equation}\label{eq-YoungIneq-conjugate}
\Phi ^*(t)
\equiv
\begin{cases}
{\displaystyle{\sup _{s\ge 0} (st - \Phi(s)),}} &
\text{when}\ t \in [0,\infty ),
\\[2ex]
\infty , &
\text{when}\ t=\infty .
\end{cases}
\end{equation}
\end{defn}

\par

\begin{rem}\label{Rem:PhiLeb}
Let $p\in [1,\infty ]$, and set
$\Phi _{[p]}(t)= t^p$ when $p \in (0,\infty)$,
and
$$
\Phi _{[\infty ]} (t) =
\begin{cases}
0, & t \leq 1,
\\[1ex]
\infty ,&  t >1.
\end{cases}
$$
Then
$L^{\Phi _{[p]}}(\rr d)$ and its 
norm is equal to the
classical Lebesgue space $L^p(\rr d)$ and its 
norm.

\par

Moreover, suppose
$p_1,p_2\in [1,\infty]$ and let
$$
\Phi (t)
=
\begin{cases}
{\frac {t^{p_2}}{p_2}}, & 0\leqslant t\leqslant 1,
\\[1ex]
{\frac {t^{p_1}}{p_1} +\frac 1{p_2}-\frac 1{p_1}},
& t> 1.
\end{cases}
$$
Here, we interpret
$\frac {t^\infty}{\infty}$
as
$$
\frac {t^\infty}{\infty}
\equiv
\lim _{p\to \infty}\frac {t^p}p
=
\begin{cases}
0, & 0\leqslant t \leqslant 1,
\\[1ex]
\infty , & t>1.
\end{cases}
$$
Then $\Phi$ is a Young function,
\begin{alignat*}{2}
L^\Phi (\rr d) &= L^{p_1}(\rr d)+L^{p_2}(\rr d),
&\quad p_1 &\leqslant p_2,
\intertext{and}
L^\Phi (\rr d)
&=
L^{p_1}(\rr d)\cap L^{p_2}(\rr d),&
\quad p_2 &\leqslant p_1.
\end{alignat*}
\end{rem}

\par





\begin{example}\label{Example:SpecYoungFunc}
The previous remark shows that sums and intersections
of Lebesgue spaces are special cases of Orlicz spaces.
Here we list some other choices of Young functions
which give rise to Orlicz spaces, where not all of
them be described by Lebesgue spaces.
\begin{itemize}
\item Let
$$
\Phi (t) =
\begin{cases}
\tan t, & 0\le t <\frac \pi 2,
\\[1ex]
\infty , & t\ge \frac \pi 2.
\end{cases}
$$
It follows that $L^\Phi =L^1\cap L^\infty$.

\vrum

\item Let
$$
\Phi (t) =
\begin{cases}
0, & t=0,
\\[1ex]
-\frac t{\ln t}, & 0< t <1,
\\[1ex]
\infty , & t\ge 1,
\end{cases}
$$
Then the conjugate Young function is given by
$$
\Phi ^*(t)
=
\left ( 
t+\frac 12 -\sqrt {\frac 14+t}\, 
\right )
e^{-\frac 1t(\frac 12+\sqrt {\frac 14+t}\, )},
$$
when $t\ge 0$ is near origin.

\vrum

\item If $\Phi (t)=t\ln{(1+t)}$, 
then $\Phi ^*(t)\asymp
\cosh{(t)}-1$.

\vrum

\item If $\Phi (t)=\cosh{(t)}-1$, 
then $\Phi ^*(t)\asymp
t\ln{(1+t)}$.
\end{itemize}

\par

We observe that each one of these Young functions
gives rise to different Orlicz spaces.

\end{example}


We refer to \cite{OsaOzt,RaoRen,HaH} for more facts about
Orlicz spaces.

\par

\subsection{Orlicz modulation spaces}\label{subsec1.4}

\par

Let $\maclM (\rr d)$ be the set of all
(complex-valued) Lebesgue measurable functions
on $\rr {d}$.
For any $p,q\in [1,\infty ]$ and $\omega \in
\mascP _E(\rr {2d})$, the norm
$\nm \cdo {M^{p,q}_{(\omega )}}$ on $\maclM (\rr {2d})$
is given by
$$
\nm F{L^{p,q}_{(\omega)}}
\equiv
\nm {H_{F,\omega ,p}}{L^q},
\qquad
H_{F,\omega ,p}(\xi )
\equiv
\nm {F(\cdo ,\xi )\cdot \omega (\cdo ,\xi )}{L^{q}},
\qquad
F\in \maclM (\rr {2d}).
$$

The definition of classical and Orlicz modulation spaces
are given in the following.
(See also \cite{Fe4,Fei5} for first definition of 
classical and more general classes of modulation spaces.)

\par

\begin{defn}\label{Def:Orliczmod}
Let $f\in \Sigma _1(\rr d)$,
$\phi(x)
=
\pi ^{-\frac{d}{4}}e^{-\frac{|x|^2}{2}},
\ x\in \rr d$,
$p,q\in [1,\infty]$, $\omega \in \mascP _E(\rr {2d})$, 
and let $\Phi$ and $\Psi$ be 
Young functions.
\begin{enumerate}
\item The \emph{modulation space} $M^{p,q}_{(\omega)}(\rr d)$
consists of all $f\in \Sigma _1' (\rr d)$ such that
\begin{equation}\label{Eq:ClassicModNorm}
\nm f{M^{p,q}_{(\omega)}}
\equiv
\nm {V_\phi}{L^{p,q}_{(\omega )}},
\end{equation}
is finite.
The topology of $M^{p,q}_{(\omega)}(\rr d)$ is
given by the norm \eqref{Eq:ClassicModNorm}.

\vrum

\item The \emph{Orlicz modulation space}
$M^{\Phi}_{(\omega )} (\rr d)$
is the set of all $f\in \maclS _{1/2}' (\rr d)$
such that
\begin{equation}\label{Eq:OrlModNorm}
\nm f{M^{\Phi}_{(\omega )}}
\equiv
\nm {V_\phi f}{L^{\Phi}_{(\omega )}}
\end{equation}
is finite. The topology of
$M^{\Phi}_{(\omega )} (\rr d)$
is given by the norm in \eqref{Eq:OrlModNorm}.
\end{enumerate}
\end{defn}

\par

Beside these well-known families
of modulation spaces, we shall also
consider the modulation space
$M^{\sharp ,q}_{(\omega )}(\rr d)$,
which consists of all
$f\in M^{\infty ,q}_{(\omega )}(\rr d)$
such that
$$
\lim _{R\to \infty}\left (
\NM {\sup _{|x|\ge R}
|V_\phi f(x,\cdo )\omega (x,\cdo )|}
{L^q}=0,
\right )
$$
when $q$, $\phi$ and $\omega$ are the same
as in Definition \ref{Def:Orliczmod}.
We notice that
$M^{\sharp ,1}_{(\omega )}(\rr d)$
is a central modulation space in
\cite{FeGaPr}, and that several
invariance properties are deduced
in \cite{NaPfTeTo}. For example it is
here shown that the window function
$\phi$ can be any element in a suitable
weighted $M^{1,1}$ class. Furthermore,
in \cite{NaPfTeTo} the following result
is obtained.

\par

\begin{lemma}
\label{Lemma:ComplModSpace}
Let $\omega \in \mascP _E(\rr {2d})$
and $q\in [1,\infty )$.
Then $M^{\sharp ,q}_{(\omega )}(\rr d)$
is the completion of $\Sigma _1(\rr d)$
under the norm
$\nm \cdo{M^{\sharp ,q}_{(\omega )}}$.
\end{lemma}

\par

For convenience we set $M^{p,p}_{(\omega )}=M^p_{(\omega )}$.
We also set
$$
M^{\Phi}_{(\omega )}=M^{\Phi},
\quad
M^{p,q}_{(\omega )}=M^{p,q}
\quad \text{and}\quad
M^{p}_{(\omega )}=M^{p}
\quad \text{when $\omega =1$ everywhere}.
$$

\par

Evidently, in Definition \ref{Def:Orliczmod}, 
we may use the transform $T_\phi$ in \eqref{Eq:TTransfDef} 
instead of $V_\phi$

\par

Let $\Phi$ be 
Young functions, and let
$\Phi _{[p]}$ be the same as
in Remark \ref{Rem:PhiLeb} and
$\omega \in \mascP _E(\rr {2d})$.
Then evidently
\begin{alignat}{3}
M^{p}_{(\omega )}(\rr d)
&=
M^{\Phi}_{(\omega )}(\rr d) &
\quad &\text{when} & \quad
\Phi &=\Phi _{[p]}.
\label{Eq:ModSpOrlModSp1}
\end{alignat}

\par

Next we explain some basic properties of
Orlicz modulation spaces. The following
proposition shows that Orlicz modulation spaces
are completely determined by the behavior
of the 
Young functions near origin. We refer to
\cite[Proposition 5.11]{ToUsNaOz} for the proof.

\par

\begin{prop}\label{Prop:OrliczModInvariance}
Let $\Phi _j$
be Young functions
and $\omega \in \mascP_E(\rr d )$.
Then the following conditions are equivalent:
\begin{enumerate}
\item $M^{\Phi _{1}}_{(\omega )}(\rr d)\subseteq
M^{\Phi _{2}}_{(\omega )}(\rr d)$;

\vrum

\item for some $t_0>0$ it holds
$\Phi _{2} (t)\lesssim  \Phi _{1} (t)$
when $t\in [0, t_0]$.
\end{enumerate}
\end{prop}

\par

The next two proposition show some other convenient
properties concerning norm invariance and
duality for Orlicz modulation spaces.
We refer to Section 4 in \cite{FG1} for their proofs.
For an exposition with Orlicz spaces in focus, see
\cite{ToUsNaOz}. In the unweighted case, some of the 
properties also follows from \cite{SchFuh}.
In the first proposition on norm invariance, we
also remark that these properties hold for a significantly
broader family of modulation spaces which also includes
more general quasi-Banach spaces (see
\cite{ToPfTe}). For classical modulation spaces,
the results can be found in Chapters 11 and 12 in
\cite{Gc2}.

\par

\begin{prop}
\label{Prop:NormInvModSp}
Let $\Phi$ be a Young function,
$\omega ,v\in \mascP  _{E}(\rr {2d})$ be such that
$\omega$ is $v$-moderate,
$\phi \in M^1_{(v)}(\rr d)\setminus 0$, and let
$f\in \Sigma _1'(\rr d)$. Then
$f\in M^\Phi _{(\omega )}(\rr d)$, if and only if
\begin{equation}
\label{Eq:ModNorm}
\nmm f \equiv \nm {V_\phi f\cdot \omega}{L^\Phi}
\end{equation}
is finite. Furthermore, $\nmm \cdo$ defines a norm
on $M^\Phi _{(\omega )}(\rr d)$ which is equivalent to
$\nm \cdo {M^\Phi _{(\omega )}}$.
\end{prop}

\par

\begin{prop}\label{Prop:BasicPropOrlModSp2}
Let $\Phi$ be a Young function, and let
$\omega \in \mascP  _{E}(\rr {2d})$.
Then the following is true:
\begin{enumerate}
\item[{\rm{(1)}}] the sesqui-linear form $( \cdo ,\cdo )_{L^2}$ on
$\Sigma _1(\rr d)$ extends to a continuous map from
$$
M^{\Phi}_{(\omega )}(\rr d)
\times
M^{\Phi ^*}_{(1/\omega )}(\rr d)
$$
to $\mathbf C$. This extension is unique when $\Phi$ and $\Psi$ fulfill
a local $\Delta _2$-condition. If $\nmm f = \sup |{(f,g)_{L^2}}|$, where
the supremum is
taken over all $b\in M^{\Phi ^*}_{(1/\omega )}(\rr d)$
such that
$\nm b{M^{\Phi ^*}_{(1/\omega )}}\le 1$, then
$\nmm {\cdot}$ and $\nm
\cdot {M^{\Phi}_{(\omega )}}$ are equivalent norms;

\vrum

\item[{\rm{(2)}}] if $\Phi$ and $\Psi$ fulfill
a local $\Delta _2$-condition, then $\Sigma _1(\rr d)$ is 
dense
in $M^{\Phi}_{(\omega )}(\rr d)$, and the dual space of
$M^{\Phi}_{(\omega
)}(\rr d)$ can be identified with
$M^{\Phi ^*}_{(1/\omega )}(\rr
d)$, through the form $(\cdo  ,\cdo )_{L^2}$.
Moreover,
$\Sigma _1(\rr d)$ is weakly dense in
$M^{\Phi ^*}_{(\omega )}(\rr
d)$.
\end{enumerate}
\end{prop}

\par






\begin{rem}
We notice that the weight classes play important
roles for the sizes of modulation spaces. More
precisely, let $\Phi$ be a Young function and
$\omega$ be a weight on $\rr {2d}$. Then
\begin{alignat*}{4}
\Sigma _s(\rr d)
&\hookrightarrow
M^\Phi _{(\omega )}(\rr d)
\hookrightarrow
\Sigma _s'(\rr d),&
\quad &\text{when}& \quad
\omega &\in \mascP _s(\rr {2d}),&
\quad s&\in (0,1],
\\[1ex]
\maclS _s(\rr d)
&\hookrightarrow
M^\Phi _{(\omega )}(\rr d)
\hookrightarrow
\maclS _s'(\rr d),&
\quad &\text{when}& \quad
\omega &\in \mascP _{0,s}(\rr {2d}),&
\quad s&\in (0,1],
\intertext{and}
\mascS (\rr d)
&\hookrightarrow
M^\Phi _{(\omega )}(\rr d)
\hookrightarrow
\mascS '(\rr d),&
\quad &\text{when}& \quad
\omega &\in \mascP (\rr {2d}).&&
\end{alignat*}
These embeddings are narrow in the sense
\begin{alignat*}{3}
\bigcap _{\omega \in \mascP _s}
M^\Phi _{(\omega )}(\rr d)
&=
\Sigma _s(\rr d),&
\qquad
\bigcup _{\omega \in \mascP _s}
M^\Phi _{(\omega )}(\rr d)
&=
\Sigma _s'(\rr d),&
\quad s&\in (0,1],
\\[1ex]
\bigcap _{\omega \in \mascP _{0,s}}
M^\Phi _{(\omega )}(\rr d)
&=
\maclS _s(\rr d),&
\qquad
\bigcup _{\omega \in \mascP _{0,s}}
M^\Phi _{(\omega )}(\rr d)
&=
\maclS _s'(\rr d),&
\quad s&\in (0,1],
\intertext{and}
\bigcap _{\omega \in \mascP}
M^\Phi _{(\omega )}(\rr d)
&=
\mascS (\rr d),&
\qquad
\bigcup _{\omega \in \mascP}
M^\Phi _{(\omega )}(\rr d)
&=
\mascS '(\rr d).& &
\end{alignat*}
(See e.{\,}g. \cite{ToUsNaOz,ToPfTe}.)
\end{rem}

\par

\begin{example}\label{Ex:OrlModSpace1}
Let $\phi \in \mascS (\rr d)\setminus 0$, and let
$\Phi$ be convex on $[0,\infty )$ such that
$$
\Phi (t)=-t^2\ln t,\quad \text{when}\quad
t\in [0,e^{-\frac 23}].
$$
Then $\Phi$ is a Young function.
The entropy functional
\begin{equation}
\label{Eq:EntropyFunc}
E_\phi (f)
\equiv
-\iint _{\rr {2d}}|V_\phi f(x,\xi )|^2
\ln |V_\phi f(x,\xi )|^2\, dxd\xi 
\end{equation}
is appears when investigating kinetic energy in
statistical physics and quantum physics, see e.{\,}g.
\cite{LieSol}.

\par

It is proved in \cite{GuRaToUs} that the following
\begin{enumerate}
\item The space $M^\Phi (\rr d)$ is close to $M^2(\rr d)$ in
the sense of the continuous inclusions
\begin{equation*}
M^p(\rr d)\subseteq M^\Phi (\rr d) \underset{\text{dense}}
\subseteq M^2(\rr d),
\qquad p<2.
\end{equation*}

\vrum

\item The functional $E_\phi $ is continuous on
$M^\Phi (\rr d)$, but fails to be continuous on
$M^2(\rr d)$.
\end{enumerate}
As a consequence of (1) and (2) one has that
$E_\phi$ is continuous on $M^p(\rr d)$ when $p<2$,
which seems not to be known before \cite{GuRaToUs}. 
\end{example}

\par

\subsection{Kernel operators and Schatten-von Neumann classes}

\par

For any topological vector spaces, $V_1$ and $V_2$, we let
$\maclL (V_1,V_2)$ be the set of all linear and continuous
operators from $V_1$ to $V_2$. Suppose
$V_1=\Sigma _1 (\rr {d_1})$,
$V_2=\Sigma _1 '(\rr {d_2})$ and
$K\in \Sigma _1 '(\rr {d_2}\times \rr {d_1})$.
Then it follows that the map $T_K$, defined by
$$
\scal {T_Kf}g = \scal K{g\otimes f},
\qquad f\in \Sigma _1 (\rr {d_1}),
\ g\in \Sigma _1(\rr {d_2}),
$$
belongs to $\maclL (V_1,V_2)$.
By the kernel theorem
of Schwartz it follows that the map
$K\mapsto T_K$, from
$\Sigma _1'(\rr {d_2}\times \rr {d_1})$ to
$\maclL (V_1,V_2)$ is bijective. For convenience
we let the topology of $\maclL (V_1,V_2)$ 
be inherited from the topology of
$\Sigma _1 '(\rr {d_2}\times \rr {d_1})$.

\par

In what follows we recall some facts on Schatten-von Neumann operators,
given in \cite{Sim}. Let $T$ be a linear and continuous map from the
Hilbert space $\maclH _1$
into the Hilbert space $\maclH _2$, {and let $j\ge 1$ be an integer.
Also let
$\mascI _{0,j}(\maclH _1,\maclH _2)$ be the set of all linear and
continuous operators from $\maclH _1$ to $\maclH _2$ with rank
at most $j-1$.
The \emph{singular value} of $T$ of order $j$
is defined by
$$
\sigma _j(T)=\sigma _j(T;\maclH _1,\maclH_2)
\equiv
\inf _{T_0\in \mascI _{0,j}}\nm {T-T_0}{\maclH _1\to \maclH _2},
\qquad j\in \mathbf Z_+ .
$$
Evidently
$\sigma _j(T;\maclH _1,\maclH_2)$ decreases with $j$, and
$\sigma _1(T;\maclH _1,\maclH_2)$ is equal to the operator
norm $\nm T{\maclH _1\to \maclH _2}$ of $T$.

\par

Throughout the paper, all Hilbert spaces are assumed to be separable,
and observe that this is always the case for Hilbert spaces which
are continuously embedded in $\mascS '(\rr d)$,
in view of \cite[Proposition 1.2]{RaToVi}.
However, we note that most parts of what is described here also
hold when $\maclH _1$ and $\maclH _2$ are allowed to be non-separable.

\par

In the following definition we present a broad family of
Schatten-von Neumann classes.

\par

\begin{defn}\label{Def:GeneralSchattenClasses}
Let $\maclH _1$, $\maclH_2$ be Hilbert spaces,
$T$ be a linear operator from $\maclH _1$ to $\maclH _2$,
and let $\maclB \subseteq \ell _0'(\mathbf Z_+)$ be a Banach space.
\begin{enumerate}
\item The $\maclB$ Schatten-von Neumann norm of $T$ is given by
$$
\nm T{\mascI _\maclB}=\nm T{\mascI _\maclB (\maclH _1,\maclH _2)}
\equiv
\nm { \{ \sigma _j(T;\maclH_1,\maclH_2)\} _{j=1}^\infty}{\maclB}.
$$

\vrum

\item The $\maclB$ Schatten-von Neumann class
$\mascI _\maclB = \mascI _\maclB (\maclH _1,\maclH _2)$
consists of all linear and continuous operators $T$ from $\maclH _1$
to $\maclH _2$ such that
$\nm T{\mascI _{\maclB} (\maclH _1,\maclH _2)}$ is finite. 
The topology
of $\mascI _\maclB = \mascI _\maclB (\maclH _1,\maclH _2)$
is given through the norm
$\nm \cdo {\mascI _\maclB (\maclH _1,\maclH _2)}$.
\end{enumerate}
\end{defn}

\par

\begin{defn}
\label{Def:OrliczSchattenClasses}
Let $\Phi$ be a Young function, $p\in [1,\infty ]$
and let $\maclH _1$ and $\maclH _2$ be Hilbert spaces.
Then let
$$
\mascI _\Phi = \mascI _{\ell ^\Phi},
\quad
\mascI _p = \mascI _{\ell ^p}
\quad \text{and}\quad
\mascI _\sharp = \mascI _{\ell ^\sharp}.
$$
\begin{itemize}
\item
The space $\mascI _\Phi (\maclH _1,\maclH _2)$
is called the Orlicz Schatten-von Neumann class with respect to
$\Phi$, $\maclH _1$ and $\maclH _2$, or the $\Phi$-Schatten class.

\vrum

\item
The  space $\mascI _p (\maclH _1,\maclH _2)$ is called the
(classical) Schatten-von Neumann class with respect to 
$p$, $\maclH _1$ and $\maclH _2$, or the $p$-Schatten class.
\end{itemize}
\end{defn}

\par

We observe that $\mascI _p(\maclH _1,\maclH _2)$ increases with
$p$, and that $\mascI _\Phi (\maclH _1,\maclH _2)$ decreases with $\Phi$.
In fact, for the latter conclusion, it suffices to detect the decreasing property
with respect to $\Phi$
near origin, which is shown in the following proposition.
The result follows from the fact that similar properties hold
true for discrete Orlicz spaces. (See e.{\,}g. \cite{BaFuTo1,SchFuh}.)
The details are left for the reader.

\par

\begin{prop}\label{Prop:SchattenEmbed}
Let $\maclH _1$ and $\maclH _2$ be Hilbert spaces, and let
$\Phi _1$ and $\Phi _2$ be Young functions such that for some
$T>0$ it holds
$$
\Phi _2(t)\lesssim \Phi _1(t),
\quad \text{when}\quad
t\in (0,T].
$$
Then
$\mascI _{\Phi _1}(\maclH _1,\maclH _2)
\hookrightarrow
\mascI _{\Phi _2}(\maclH _1,\maclH _2)$.
\end{prop}

%

\par

We notice that
$$
\mascI _1(\maclH _1,\maclH _2),
\quad
\mascI _2(\maclH _1,\maclH _2),
\quad \text{and}\quad
\mascI _\infty(\maclH _1,\maclH _2),
$$
are the spaces of trace-class, Hilbert-Schmidt, and linear
and continuous operators from $\maclH _1$ to $\maclH _2$, respectively,
also in norms.

\par

By straight-forward application of the spectral theorem, it follows
that the definition of $\mascI _p(\maclH )$, $\mascI _\Phi (\maclH )$
and their norms can be reached with other approaches, indicated in
the following.

\par

\begin{prop}
\label{Prop:OrlSchattTop}
Let $\maclH _1$ and $\maclH _2$ be Hilbert spaces,
and $\Phi$ be a Young function.
Then $\mascI _\Phi (\maclH _1,\maclH _2)$ is a Banach, and
\begin{equation}
\label{Eq:AltOrlSchattNorm}
\nm T{\mascI _\Phi}
=
\sup \nm { \{ (Tf_j,g_j)_{\maclH _2}\}_{j=1}^\infty}{\ell ^\Phi (\mathbf Z_+)},
\quad
T\in \mascI _\infty (\maclH _1,\maclH _2).
\end{equation}
Here the supremum is taken over all orthonormal sequences
$\{ f_j\} _{j=1}^\infty \in \ON (\maclH _1)$
and
$\{ g_j\} _{j=1}^\infty \in \ON (\maclH _2)$.
\end{prop}

\par

Since our investigations especially concerns Schatten-von Neumann properties
for operators acting on Hilbert modulation spaces, it is convenient for us
to set
$$
\mascI _\Phi (\omega _1,\omega _2)
=
\mascI _\Phi (M^2_{(\omega _1)}(\rr {d_1}),M^2_{(\omega _2)}(\rr {d_2}))
$$
and
$$
\mascI _p (\omega _1,\omega _2)
=
\mascI _p (M^2_{(\omega _1)}(\rr {d_1}),M^2_{(\omega _2)}(\rr {d_2})).
$$

\section{Continuity for Fourier integral operators}\label{sec2} 

\par

In this section we obtain continuity for Fourier integral
operators when acting on Orlicz modulation spaces.
In the first part (Subsection \ref{subsec2.1})
we deduce continuity properties for operators with
kernels belonging to Orlicz modulation spaces.
Thereafter we consider Fourier integral operators
with amplitudes belonging to $M^{\infty ,1}_{(\omega )}$
and the phase function $\fy$
should satisfy \eqref{Eq:DetphiCond} and
$\fy ''\in M^{\infty ,1}_{(v)}$. Here $\omega$ is
$v$-moderate. We show that such Fourier integral
operators are continuous from $M^\Phi _{(\omega _1)}$
to $M^\Phi _{(\omega _2)}$, provided the weight functions
$\omega$, $\omega _1$ and $\omega _2$
obey suitable estimate conditions.

\par

Thereafter we consider Fourier integral
operators with amplitudes belonging to
the Orlicz modulation space $M^\Phi _{(\omega )}$, and
where the condition \eqref{Eq:DetphiCond} is replaced
by other ones. We show that such operators are continuous
from $M^{\Phi ^*} _{(\omega _1)}$ to
$M^\Phi _{(\omega _2)}$.

\par

\subsection{Mapping properties
for certain kernel operators
on Orlicz modulation spaces}
\label{subsec2.1}

\par

The next result concerns mapping properties for operators with
kernels in Orlicz modulation spaces.

\par

\begin{prop}
\label{Prop:KernelOrlModCont}
Let $\Phi$ be a Young function such that at least one of $\Phi$
and $\Phi ^*$ satisfies the
$\Delta _2$-condition, let 
$\omega \in \mascP _E(\rr {2d})$
and let $\omega _j\in \mascP _E(\rr {d_j})$,
$j=1,2$, be such that
$$
\frac {\omega _2(x_2,\xi _2)}{\omega _1(x_1,-\xi _1)}\lesssim \omega (x_1,x_2,\xi _1,\xi _2),
\qquad
x_j,\xi _j\in \rr {d_j}.
$$
Also let $K\in M^\Phi _{(\omega )}(\rr {d_2+d_1})$. 
Then $T_K$ from $\Sigma _1 (\rr {d_1})$
to $\Sigma _1 '(\rr {d_2})$ extends uniquely to a 
continuous operator from
$M^{\Phi ^*}_{(\omega _1)}(\rr {d_1})$ to 
$M^{\Phi}_{(\omega _2)}(\rr {d_2})$,
\begin{equation}
\label{Eq:KernelOrlModCont}
\begin{aligned}
\nm {T_Kf}{M^\Phi _{(\omega _2)}}
&\le
C\nm K{M^\Phi _{(\omega )}}\nm f{M^{\Phi ^*}_{(\omega _1)}},
\\[1ex]
K&\in M^\Phi _{(\omega )}(\rr {d_2+d_1}),\quad
f\in M^{\Phi ^*}_{(\omega _1)}(\rr {d_1}).
\end{aligned}
\end{equation}
\end{prop}

\par

The result follows by standard arguments in operator
theory.
In order to assist the reader we here present the arguments.

\par

\begin{proof}
First suppose $\Phi$ satisfies the $\Delta _2$-
condition. 
Let $\phi _j\in \Sigma _1 (\rr {d_j})\setminus 0$ 
and let
$\phi (x,y) =\phi _2(x)\overline {\phi _1(y)}$. 
Since $\mascS (\rr {d_2+d_1})$
is dense in $M^\Phi _{(\omega )}(\rr {d_2+d_1})$,
it suffices to prove
\eqref{Eq:KernelOrlModCont} when $K\in
\Sigma _1 (\rr {d_2+d_1})$.
Let
$$
\maclK (x,\xi ,y,\eta ) = V_\phi K(x,y,\xi ,-\eta )\omega (x,y,\xi ,-\eta )
\quad \text{and}\quad F =V_{\phi _1}f \omega _1.
$$
Then
$$
|V_{\phi _2}(T_Kf)(x,\xi )\omega _2(x,\xi )|
\lesssim
|(\maclK (x,\xi ,\cdo ),F))_{L^2(\rr {2d_2})}|
\lesssim
\nm {\maclK (x,\xi ,\cdo )}{L^\Phi}\nm F{L^{\Phi ^*}}.
$$
This gives
$$
\nm {T_Kf}{M^\Phi _{(\omega _2)}}
\lesssim
\nm {\maclK}{L^\Phi}\nm F{L^{\Phi ^*}}
\asymp
\nm K{M^\Phi _{(\omega )}}\nm f{M^{\Phi ^*}_{(\omega _1)}}.
$$

\par

The
case when $\Phi ^*$ satisfies the $\Delta _2$-condition follows by
straight-forward modification of the arguments, giving the estimate
$$
|(T_Kf,g)_{L^2(\rr {d_2})}|\lesssim
\nm K{M^\Phi _{(\omega )}}\nm f{M^{\Phi ^*}_{(\omega _1)}}
\nm g{M^{\Phi ^*}_{(1/\omega _2)}},
$$
when $f\in \Sigma _1 (\rr {d_1})$
and $g\in \Sigma _1(\rr {d_2})$,
and then using the fact that $\mascS$ is dense in $M^{\Phi ^*}_{(\omega _1)}$
and in $M^{\Phi ^*}_{(1/\omega _2)}$.
The details are left for the reader.
\end{proof}

\par

\subsection{Continuity for Fourier
integral operators, when acting
on modulation spaces}

\par

In order to explain our results on
Fourier integral operators, we first
give the conditions on involved weight
functions.
We usually assume that they satisfy
\begin{equation}
\label{Eq:WeightsIneq}
\begin{aligned}
\frac {\omega _2(x,\xi )}{\omega _1(y,-\eta )}
&\le
C_1\omega _0(x,y,\xi ,\eta )
\\
&\le C_2\omega (X,\xi
-\fy '_x(X) , \eta -\fy '_y(X),-\fy '_\zeta (X)),
\\[1ex]
\omega (X,\xi _1+&\xi _2,\eta _1+\eta _2,z_1+z_2) \le C\omega (X,\xi
_1,\eta _1,z_1)v_0(\xi _2,\eta _2,z_2),
\\[1ex]
\omega &\in \mascP _s(\rr {2(2N+m)}),
\quad \omega _0 \in \mascP _s(\rr {2(2N)}),
\\[1ex]
\omega _j&\in \mascP _s(\rr {2d_j}),\quad 
N=d_1+d_2,\ j=1,2,
\\[1ex]
v_0&\in \mascP _s(\rr {2N+m}),\quad
\sup _{t\in [0,1]}
\nm {v_0(t\cdo )/v_0}{L^\infty}<\infty ,
\\[1ex]
v(X,\xi ,\eta ,z) &=v_0(\xi ,\eta ,z),
\ X=(x,y,\zeta )\in \rr {2N+m},
\end{aligned}
\end{equation}
and that the phase function should satisfy
\begin{equation}\label{Eq:PhaseFuncCond}
\fy \in C(\rr {2d+m})
\quad \text{and}\quad
\fy ^{(\alpha )}\in M^{\infty ,1}_{(v)}
(\rr {2d+m}),\ |\alpha |=2.
\end{equation}

\par

The first result shows that
Fourier integral operators with
amplitudes in $M^{\infty ,1}_{(\omega)}
(\rr {2d+N})$ are well-defined as
continuous mappings from
$\Sigma _1(\rr d)$ to $\Sigma _1'(\rr d)$.
For polynomial weights, the result is 
essentially the same as
\cite[Theorem 2.1]{ToCoGa}.
For general weights, the result
follows by similar arguments as for
\cite[Theorem 2.1]{ToCoGa}.
The proof is therefore omitted.

\par




\begin{thm}\label{Thm:Cont1Prel}
Let $s>1$, $d_1=d_2=d$, 
$\fy \in C(\rr {2d+m})$, and
$\omega ,v\in \mascP _s(\rr {2(2d+m)})$ 
satisfy
\eqref{Eq:DetphiCond},
\eqref{Eq:WeightsIneq} and
\eqref{Eq:PhaseFuncCond}. Then the 
map $a\mapsto \op _\fy (a)$
from $\Sigma _1(\rr {2d+m})$ to
$\maclL (\Sigma _1(\rr d),\Sigma _1'(\rr d))$ extends uniquely to
a continuous map from $M^{\infty ,1}_{(\omega )}(\rr {2d+m})$ to
$\maclL (\Sigma _1(\rr d),\Sigma _1'(\rr d))$.
\end{thm}

\par

The next two theorems assert continuity
and compactness for Fourier integral
operators in Theorem \ref{Thm:Cont1Prel},
when they are acting on Orlicz modulation
spaces.

\par

\begin{thm}\label{Thm:Cont1}
Let $s>1$, $d_1=d_2=d$, $\Phi$
be a Young function such that
$1<q_\Phi \le p_\Phi <\infty$,
$\omega _1,\omega _2\in
\mascP _s(\rr {2d})$, and let
$\fy$, $\omega$, $v$ be the same
as in Theorem \ref{Thm:Cont1Prel}.
Also let $a\in M^{\infty ,1}_{(\omega )}
(\rr {2d+m})$.
Then $\op _{\fy }(a)$ from
$\Sigma _1(\rr d)$
to $\Sigma _1'(\rr d)$ extends
uniquely to a continuous operator from
$M^\Phi _{(\omega _1)}(\rr d)$ to
$M^\Phi _{(\omega _2)}(\rr d)$.
Moreover, for some constant $C$ it holds
\begin{equation}\label{normuppsk2}
\nm {\op _{\fy }(a)}
{M^\Phi _{(\omega _1)}\to M^\Phi _{(\omega _2)}}
\le
C\dbar ^{\, -1}\nm a{M^{\infty ,1}_{(\omega )}}
\exp (C\nm {\fy ''}{M^{\infty,1}_{(v)}}).
\end{equation}
\end{thm}

\par


\par

\begin{thm}\label{Thm:Cont1Comp}
Let $s>1$, $d_1=d_2=d$, $\Phi$
be a Young function such that
$1<q_\Phi \le p_\Phi <\infty$,
$\omega _1,\omega _2\in
\mascP _s(\rr {2d})$, and let
$\fy$, $\omega$, $v$ be the same
as in Theorem \ref{Thm:Cont1Prel}.
Also let $a\in M^{\sharp ,1}_{(\omega )}(\rr 
{2d+m})$. Then $\op _{\fy }(a)$ 
from
$M^\Phi _{(\omega _1)}(\rr d)$ to
$M^\Phi _{(\omega _2)}(\rr d)$
is compact.
\end{thm}

\par

Theorem \ref{Thm:Cont1Comp} will follow
by combining Theorem \ref{Thm:Cont1},
with certain density arguments and the 
following proposition.

\par

\begin{prop}\label{Prop:Cont1Comp}
Let $s>1$, $\Phi$,
$\fy$, $\omega$, $v$, $\omega
_1$, $\omega _2$ be the same
as in Theorem \ref{Thm:Cont1}, and
let $a\in \Sigma _1(\rr 
{2d+m})$. Then $\op _{\fy }(a)$ 
from
$M^\Phi _{(\omega _1)}(\rr d)$ to
$M^\Phi _{(\omega _2)}(\rr d)$
is compact.
\end{prop}

\par

Before proving
Theorems \ref{Thm:Cont1}, 
\ref{Thm:Cont1Comp}, and
Proposition \ref{Prop:Cont1Comp},
we present some further continuity 
properties for Fourier integral
operators.
The amplitudes should
fulfill norm estimates, with
norm $\nmm \cdo$ either defined by
\begin{equation}
\label{Eq:NormSymb1}
\begin{aligned}
\nmm a &= 
\NM{\int _{\rr m}H_{a,\omega }(\cdo ,\zeta)\, d\zeta}{L^\Phi},
\\[1ex]
H_{a,\omega} (X)
&=
\sup _{z \in \rr m}\left (
\nm {V_\phi a(X,\cdo ,z)\omega (X,\cdo ,z)}{L^\Phi}
\right ),
\quad
X=(x,y,\zeta),
\end{aligned}
\end{equation}
or
\begin{equation}
\label{Eq:NormSymb2}
\begin{aligned}
\nmm a &= 
\NM{\int _{\rr m}H_{a,\omega }(\cdo ,z)\, dz}{L^\Phi},
\\[1ex]
H_{a,\omega} (x,y,z)
&=
\sup _{\zeta \in \rr m}\left (
\nm {V_\phi a(X,\cdo ,z)\omega (X,\cdo ,z)}{L^\Phi}
\right ),
\quad
X=(x,y,\zeta),
\end{aligned}
\end{equation}
or by
\begin{equation}
\label{Eq:NormSymb3}
\begin{aligned}
\nmm a &= \int _{\rr m}
\nm{H_{a,\omega }(\cdo ,z)}{L^\Phi}
\, dz,
\\[1ex]
H_{a,\omega}
(x,y,\xi ,\eta ,z)
&=
\sup _{\zeta \in \rr m}\left (
|V_\phi a(X,\xi ,\eta ,z)\omega (X,\xi ,\eta ,z)|
\right ),
\quad
X=(x,y,\zeta).
\end{aligned}
\end{equation}
We search estimates of the form
\begin{align}
\nm {K_{a,\fy}}{M^\Phi _{(\omega _0)}}
\le C{\dbar}^{-1}\exp \left ( \nm {\fy ''}{M^{\infty ,1}_{(v)}}\right )
\nmm a ,
\label{Eq:KernelEst}
\intertext{and}
\nm {\op _\fy (a)}{M^{\Phi ^*}_{(\omega _1)}\to M^{\Phi}_{(\omega _2)}}
\le
C{\dbar}^{-1}\exp \left ( \nm {\fy ''}{M^{\infty ,1}_{(v)}}\right )\nmm a .
\label{Eq:OpEst}
\end{align}

\par

\begin{thm}\label{Thm:Cont2}
Let $s>1$, $N=d_2+d_1$, $\Phi$ be a Young function such 
that $1<q_\Phi \le
p_\Phi <\infty$, $q_\Phi =p_\Phi =1$ or $q_\Phi 
=p_\Phi =\infty$, let
$\phi \in \Sigma _s(\rr {N+m})\setminus 0$,
$\omega$, $\omega _j$, $j=0,1,2$, $v$ and 
$\fy$ be as in
\eqref{Eq:WeightsIneq} and 
\eqref{Eq:PhaseFuncCond}.
Also let $a\in \Sigma _s'(\rr {N+m})$
be such that one of the
following conditions holds:
\begin{itemize}
\item[{\rm{(i)}}] $\nmm a$
in \eqref{Eq:NormSymb1} is finite;

\vrum

\item[{\rm{(ii)}}] $\nmm a$
in \eqref{Eq:NormSymb2} is finite
and that $|\det (\fy ''_{\zeta ,\zeta})|
\ge \dbar$ for some $\dbar >0$.
\end{itemize}
Then the following is true:
\begin{enumerate}
\item the kernel $K_{a,\fy}$ of $\op _\fy (a)$ belongs to
$M^\Phi _{(\omega _0)}(\rr d)$, and \eqref{Eq:KernelEst} holds
for some constant $C$ which is
independent of $a\in \Sigma _s '(\rr {N+m})$
and $\fy \in C (\rr {N+m})$;

\vrum

\item the definition of $\op _\fy (a)$ extends uniquely to a continuous operator from
$M^{\Phi ^*}_{(\omega _1)}(\rr {d_1})$ to $M^{\Phi}_{(\omega _2)}(\rr {d_2})$, and
\eqref{Eq:OpEst} holds.
\end{enumerate}
\end{thm}

\par

For the proofs of these results
we need the following lemma. Here we 
formulate
the action of $\op _\fy (a)$ as
\begin{align}
(\op _\fy (a)f_1,f_2)
&=
\iiint _{\rr {4N+m}}\maclK _{a,\fy}(X,\xi ,\eta )F_1(y,\eta )\overline{F_2(x,\xi )}
e^{-i(\scal x\xi +\scal y\eta )}\, dX d\xi d\eta ,
\label{Eq:FIOAction1}
\intertext{with}
F_1(y,\eta )
&=
V_{\phi _1}f_1(y,-\eta )\omega _1(y,-\eta ),
\ F_2(x,\xi )
=
V_{\phi _2}f_2(x,\xi )/\omega _2(x,\xi )
\label{Eq:FIOAction2}
\end{align}

\par

\begin{lemma}\label{Lem:LemmaThmCont2}
Let $s>1$, $\Phi$ and $\Psi$ be Young functions, 
$N=d_1+d_2$, 
$\omega ,v \in \mascP _s(\rr {2(N+m)})$ be as in
\eqref{Eq:WeightsIneq}, $f_j\in \Sigma _1
(\rr {d_j})$ 
and $\phi \in C_0^\infty (\rr {N+m})\setminus 0$.
Also let $a\in M^{\Phi ,\Psi}_{(\omega )}(\rr {N+m})$
or $a\in M^{\Phi}_{(\omega )}(\rr {N+m})$, and
$f_j\in \mascS (\rr {d_j})$,
$j=1,2$. Then \eqref{Eq:FIOAction1} and 
\eqref{Eq:FIOAction2} hold for some
$\maclK _{a,\fy}$, which satisfies
\begin{multline*}
|\maclK _{a,\fy} (X,\xi ,\eta )|
\le
(G*|V_\phi a(X ,\cdo )\omega (X,\cdo )|)
(\xi -\fy '_x(X),\eta -\fy '_y(X),-\fy '_\zeta (X)),
\end{multline*}
for some non-negative $G\in L^1(\rr {N+m})$
which satisfies
\begin{equation}\label{Eq:GL1Est}
\nm G{L^1}\le C\exp (C\nm {\fy ''}{M^{\infty ,1}_{(v)}}).
\end{equation}
\end{lemma}

\par

\begin{proof}
Let
$$
\Theta =(\xi -\fy '_x(X),\eta -\fy '_y(X),-\fy '_\zeta (X)),
$$
and choose $\phi _j\in \maclD _{0,s} (\rr {d_j})$
such that
\begin{equation}
\label{Eq:CouplWindFunc}
\begin{alignedat}{2}
0 &\le \phi _j ,\phi , &
\quad
\int _{\rr {d_j}}\phi _j(x_j)\, dx_j
&=
\int _{\rr {N+m}} \phi (X)\phi _1(y)
\phi _2(x)\, dX =1,
\\[1ex]
j&=1,2, &\qquad  X &= (x,y,\zeta )\in \rr {N+m}.
\end{alignedat}
\end{equation}
By (2.5), (2.15), Lemma 2.2, Lemma 2.3
in \cite{ToCoGa},
and straight-forward computations, it follows that \eqref{Eq:FIOAction1}
holds with $\maclK _{a,\fy}^0$,
$$
F_{0,1}(y,\eta )
=
V_{\phi _1}f_1(y,-\eta )
\quad \text{and}\quad F_{0,2}(x,\xi )
=
V_{\phi _2}f_2(x,\xi )
$$
in place of $\maclK _{a,\fy}$, $F_1$ and $F_2$, where $\maclK _{a,\fy}^0$
satisfies
$$
|\maclK _{a,\fy} ^0(X,\xi ,\eta )|
\le
(G_0*|V_\phi a(X ,\cdo )|) (\Theta ),
$$
with $G_0\in L^1_{(v_0)}(\rr {d+m})$ satisfying
$$
\nm {G_0}{L^1_{(v_0)}}\le C\exp (C\nm {\fy ''}{M^{\infty ,1}_{(v)}}).
$$
The details are left for the reader. Then \eqref{Eq:FIOAction1}
holds if
$$
\maclK _{a,\fy} (X,\xi ,\eta )
=
\maclK _{a,\fy} ^0(X,\xi ,\eta )\omega _2(x,\xi )/ \omega _1(y,-\eta).
$$

\par

A combination of the latter relationships with \eqref{Eq:WeightsIneq} gives
\begin{align*}
|\maclK _{a,\fy} (X,\xi ,\eta )|
&\le
(G_0*|V_\phi a(X ,\cdo )|) (\Theta )\omega _2(x,\xi )/ \omega _1(y,-\eta)
\\[1ex]
&\lesssim
(G_0*|V_\phi a(X ,\cdo )|) (\Theta )\omega (X,\Theta)
\\[1ex]
&\lesssim
(G*|V_\phi a(X ,\cdo )\omega (X,\cdo )|) (\Theta ),
\end{align*}
where $G=G_0v_0$ satisfies \eqref{Eq:GL1Est}. This gives the
result.
\end{proof}

\par

\begin{rem}
Let $a$, $\phi$, $\phi _j$, $\fy$ and $\Theta$
be the same as in Lemma \ref{Lem:LemmaThmCont2}
and its proof, let $v$ and $v_0$ be as in
\eqref{Eq:WeightsIneq} and let
$$
v_1(x,y,\zeta ,\Xi )=v_0(x,y,\zeta ).
$$
Then \cite[Lemma 2.2]{ToCoGa} shows that
\begin{align}
\maclK _{a,\fy}(X,\xi ,\eta )\cdot \frac {\omega _1(y,-\eta )}{\omega _2(x,\xi )}
&=
(2\pi )^{\frac m2}
((\mascF (e^{i\fy _{2,X}}\phi ))*(T_\phi a))
(\Theta )
\\[1ex]
\fy _{2,X}
&=
\psi (X)\int _0^1(1-t)
\scal {\fy ''(X+tY)Y}Y\, dt,
\\[1ex]
\psi
&\in
C_0^\infty (\rr {d+m}),
\quad
\psi =1\ \text{on}\ \supp \phi .
\end{align}
Here $T_\phi$ is given by \eqref{Eq:TTransfDef}.
The function $G_0$ in the proof of Lemma 
\ref{Lem:LemmaThmCont2} can then be chosen
as
$$
G_0 =
|\mascF (e^{i\fy _{2,X}}\phi )|
$$

\par

By \cite[Lemma 2.3]{ToCoGa} one has
$$
|\mascF (e^{i\fy _{2,X}})|
\le
(2\pi )^{\frac 12(d+m)}
\delta _0+\hat b,
$$
where $b$ satisfies
$$
\nm b{M^1_{(v)}} \le \exp
\left (
C\nm {\fy ''}{M^{\infty ,1}_{(v)}}
\right )
$$
It follows that
$$
G_0 =
|\mascF (e^{i\fy _{2,X}}\phi )|
=(2\pi )^{-\frac 12(d+m)}
|\mascF (e^{i\fy _{2,X}})*\widehat \phi |
\lesssim
|\widehat \phi |+|\widehat b*\widehat \phi |.
$$
Since
$$
\widehat b\in M^1_{(v_1)}(\rr {d+m}) \subseteq
L^1_{(v_0)}(\rr {d+m})
\quad \text{and}\quad
\widehat \phi \in \mascS (\rr {d+m})
\subseteq L^1_{(v_0)}(\rr {d+m}),
$$
it follows that $G_0\in L^1_{(v_0)}(\rr {d+m})$,
as it was used in the proof of Lemma
\ref{Lem:LemmaThmCont2}.
\end{rem}

\par

\begin{rem}
\label{Rem:LemmaThmCont2}
Let $\phi$ and $\phi _j$ be as in Lemma
\ref{Lem:LemmaThmCont2}, and choose $\psi _j\in \maclD _{0,s}(\rr {d_j})$
such that $(\phi _j,\psi _j)=1$, $j=1,2$. Also let $\psi =\psi _2\otimes \psi _1$.
Then it follows from \eqref{Eq:FIOAction1} and an application
of Fourier's inversion formula that
\begin{multline}
\label{Eq:KernelRel}
V_\psi K_{a,\fy}(x,y,\xi ,\eta )\cdot \frac {\omega _2(x,\xi )}{\omega _1(y,-\eta )}
\\[1ex]
=
Ce^{-i(\scal x\xi +\scal y\eta )}\int _{\rr m}\maclK _{a,\fy}(X,\xi ,\eta )\, d\zeta ,
\quad X=(x,y,\zeta ).
\end{multline}
\end{rem}

\par

\begin{proof}[Proof of Theorem \ref{Thm:Cont1}]
We shall follow the proof of \cite[Theorem 2.1]{ToCoGa} when
proving (2). If $\Phi (t_0)=0$ for some $t_0>0$, then
$M^\Phi _{(\omega _j)}(\rr d)=
M^\infty _{(\omega _j)}(\rr d)$, and the result follows 
from \cite[Theorem 2.1]{ToCoGa}. Therefore suppose that
$\Phi (t)>0$ when $t>0$. Then $\mascS (\rr d)$ is dense
in $M^\Phi _{(\omega _j)}(\rr d)$.

\par

Suppose that $f_1,f_2\in \mascS (\rr d)$ satisfy
\begin{equation}
\label{Eq:Normalizedfj}
\nm {f_1}{M^\Phi _{(\omega _1)}}
=
\nm {f_2}{M^\Psi _{(\omega _2)}}
=1
\end{equation}
Then Lemma \ref{Lem:LemmaThmCont2} gives
\begin{align*}
|(\op _\fy (a)f_1,f_2)|
&\le
\iiint _{\rr {4d+m}}(G*H_{a,\omega})(\Theta )
|F_1(y,\eta )|\, |F_2(x,\xi )|\, dXd\xi d\eta ,
\intertext{where}
H_{a,\omega}
&=
\sup _{X\in \rr {4d+m}}|V_\phi a(X,\cdo ) \omega (X,\cdo )|
\intertext{and}
\Theta
&=
(\xi -\fy '_x(X),\eta -\fy '_y(X),-\fy '_\zeta (X)),
\quad
X=(x,y,\zeta ).
\end{align*}

\par

By letting $t_1=|F_1(y,\eta )|$ and $t_2=|F_2(x,\xi )|$
in $t_1t_2\le \Phi (t_1)+\Psi (t_2)$, we obtain
\begin{align}
|(\op _\fy (a)f_1,f_2)|
&\le
J_1+J_2,
\label{Eq:OpActionSplit}
\intertext{where}
J_1
&=
\iiint _{\rr {4d+m}}(G*H_{a,\omega})(\Theta )
\Phi (|F_1(y,\eta )|)\, dXd\xi d\eta 
\notag
\intertext{and}
J_2
&=
\iiint _{\rr {4d+m}}(G*H_{a,\omega})(\Theta )
\Psi (|F_2(x,\xi )|)\, dXd\xi d\eta 
\notag
\end{align}

\par

We need to estimate $J_1$ and $J_2$. By taking
$z=\fy '_\zeta (X)$, $\eta _0$, $y$, $\xi$ and $\eta$
as new variables of integrations,
and using \eqref{Eq:DetphiCond}, it follows that
\begin{align*}
J_1
&\le
\dbar ^{\, -1}
\iiint _{\rr {4d+m}}(G*H_{a,\omega})
(\xi -\kappa (y,z,\eta _0),\eta -\eta _0,z)
\Phi (|F_1(y,\eta )|)
\, dy\,  dz\,  d\xi\,  d\eta\,  d\eta _0 
\\[1ex]
&=
\dbar ^{\, -1}
\iiint _{\rr {4d+m}}(G*H_{a,\omega})
(\xi ,\eta -\eta _0,z)
\Phi (|F_1(y,\eta )|)\, dy\, dz\, d\xi\, d\eta\, d\eta _0 
\\[1ex]
&=
\dbar ^{\, -1}
\nm {G*H_{a,\omega}}{L^1}
\iint \Phi (|F_1(y,\eta )|)\, dy d\eta 
\le
C\dbar ^{\, -1}
\nm {G*H_{a,\omega}}{L^1}
\end{align*}
for some continuous function $\kappa$ and constant
$C>0$.
In the last inequality we have used the fact that
$$
\iint \Phi (|F_1(y,\eta )|)\, dy d\eta \le 1
$$
when $\nm {f_1}{M^\Phi _{(\omega _1)}}=1$.
It follows by Young's
inequality, \eqref{Eq:WeightsIneq} and Lemma
\ref{Lem:LemmaThmCont2}
that
$$
\nm {G*H_{a,\omega}}{L^1}
\le 
\nm G{L^1}\nm {H_{a,\omega}}{L^1}
\le
C\nm a{M^{\infty ,1}_{(\omega )}}
\exp (C\nm {\fy ''}{M^{\infty ,1}_{(v)}}),
$$
for some constant $C>0$. Hence
\begin{align}
J_1 &\le C\dbar ^{\, -1}\nm a{M^{\infty ,1}_{(\omega )}}
\exp (C\nm {\fy ''}{M^{\infty ,1}_{(v)}}),
\label{Eq:J1EstThm1}
\intertext{for some constant $C>0$.
\newline
\indent
If we instead take $x$, $z=\fy '_\zeta (X)$,
$\xi$, $\eta$ and $\xi _0=\fy '_x(X)$ as new
variables of integrations, it follows by
similar arguments that}
J_2 &\le C\dbar ^{\, -1}\nm a{M^{\infty ,1}_{(\omega )}}
\exp (C\nm {\fy ''}{M^{\infty ,1}_{(v)}}),
\label{Eq:J2EstThm1}
\end{align}


A combination of \eqref{Eq:OpActionSplit},
\eqref{Eq:J1EstThm1} and \eqref{Eq:J2EstThm1}
now gives
$$
|(\op _\fy (a)f_1,f_2)|
\le
C\dbar ^{\, -1}\nm a{M^{\infty ,1}_{(\omega )}}
\exp (C\nm {\fy ''}{M^{\infty ,1}_{(v)}})
$$
when \eqref{Eq:Normalizedfj} holds.
The result now follows from this estimate,
by homogeneity, duality and the fact that
$\mascS (\rr d)$ is dense in
$M^\Phi _{(\omega _1)}(\rr d)$.
\end{proof}

\par

\begin{proof}[Proof of Proposition
\ref{Prop:Cont1Comp}]
Let $X=(x,y,\zeta )$,
\begin{align*}
\vartheta (X,\xi ,\eta ,z)
&=
\omega (X,\xi ,\eta ,z)
(1+|x|+|y|+|z|+|\xi |)^3
\intertext{and}
\vartheta _2(x,\xi )
&=
\omega _2(x,\xi )(1+|x|)(1+|\xi |).
\end{align*}
Since $|\fy ''(x,y,\zeta )|$ is a
bounded function, it follows that
$$
|\fy '_x|=|\fy '_x(X)|
\lesssim
(1+|x|+|y|+|\zeta |).
$$
This gives
\begin{align*}
&\omega
(X,\xi -\fy '_x,\eta -\fy '_y,
-\fy '_\zeta )(1+|x|)(1+|\xi |)
\\[1ex]
&\le
\omega
(X,\xi -\fy '_x,\eta -\fy '_y,
-\fy '_\zeta )(1+|x|)(1+|\xi -\fy '_x|)
(1+|\fy '_x |)
\\[1ex]
&\lesssim
\omega
(X,\xi -\fy '_x,\eta -\fy '_y,
-\fy '_\zeta )(1+|x|+|y|+|\zeta |+
|\xi -\fy '_x|)^3
\\[1ex]
&=
\vartheta
(X,\xi -\fy '_x,\eta -\fy '_y,
-\fy '_\zeta ).
\end{align*}
A combination of these estimates and
\eqref{Eq:WeightsIneq} shows that
\eqref{Eq:WeightsIneq} holds with
$\vartheta$ and $\vartheta _2$ in place
of $\vartheta$ and $\vartheta _2$.

\par

Since $a\in \Sigma _1(\rr {2d+N})$,
it follows that $a\in M^{\infty 
,1}_{(\vartheta )}(\rr {2d+N})$.
Hence Theorem \ref{Thm:Cont1}
shows that 
\begin{equation}
\label{Eq:ContOfFIO1}
\op _\fy (a) :
M^\Phi _{(\omega _1)}(\rr d)
\to
M^\Phi _{(\vartheta _2)}(\rr d)
\end{equation}
is continuous.

\par

Since the inclusion map 
\begin{equation}
\label{Eq:CompOfFIO1}
\boldsymbol \iota
: M^\Phi _{(\vartheta _2)}(\rr d)
\to
M^\Phi _{(\omega _2)}(\rr d),
\end{equation}
is compact, in view of
\cite[Proposition 4.8]{ToPfTe},
it follows that
$$
\op _\fy (a) =\boldsymbol \iota
\circ
\op _\fy (a) :
M^\Phi _{(\omega _1)}(\rr d)
\to
M^\Phi _{(\omega _2)}(\rr d)
$$
is compact, giving the result.



\end{proof}

\par

\begin{proof}[Proof of Theorem
\ref{Thm:Cont1Comp}]
By Theorem \ref{Thm:Cont1} we have
$$
\nm {\op _\fy (a)}{M^\Phi _{(\omega _1)}
\to M^\Phi _{(\omega _2)}}
\lesssim
\nm {a}{M^{\infty ,1}_{(\omega )}}.
$$
Since $a\in M^{\sharp ,1}_{(\omega )}
(\rr {2d+N})$, it follows from
Lemma \ref{Lemma:ComplModSpace}
and Proposition
\ref{Prop:Cont1Comp},
that for some $a_j\in \Sigma _1(\rr 
{2d+N})$, one has
$$
\nm {a-a_j}{M^{\infty ,1}_{(\omega )}}
\to 0,
\quad \text{as}\quad j\to \infty .
$$
A combination of these estimates gives
\begin{align*}
\nm {\op _\fy (a)-\op _\fy (a)}
{M^\Phi _{(\omega _1)}
\to M^\Phi _{(\omega _2)}}
\lesssim
\nm {a-a_j}{M^{\infty ,1}_{(\omega )}}
\to 0,
\quad \text{as}\quad j\to \infty .
\end{align*}
Hence, $\op _\fy (a)$ can be approximated
in norm by $\op _\fy (a_j)$. Since
$a_j\in \Sigma _1(\rr {2d+N})$,
it follows that $\op _\fy (a_j)$. 
A combination of these properties shows
that $\op _\fy (a)$ is compact,
and the result follows.
\end{proof}

\par


\par

%

\begin{proof}[Proof of Theorem \ref{Thm:Cont2}]
The result follows from \cite[Theorem 2.7]{ToCoGa}
when $p_\Phi =q_\Phi$. Therefore we assume that
$1<q_\Phi \le p_\Phi <\infty$.

First we suppose that (i) holds. Let $\phi$, $\phi _j$ and $G$ be as in
Lemma \ref{Lem:LemmaThmCont2}, and let
$$
U=|V_\phi a \cdot \omega | .
$$
By Lemma \ref{Lem:LemmaThmCont2},
Remark \ref{Rem:LemmaThmCont2} and Minkowski's inequality we obtain
\begin{align}
\nm {K_{a,\fy}}{M^\Phi}
&\lesssim
\NM {\int _{\rr m}|\maclK _{a,\fy} (X,\xi ,\eta )|\, d\zeta }{L^\Phi}
\notag
\\[1ex]
&\lesssim
\NM {\int _{\rr m}(G*U(X,\cdo ))(\Theta )\, d\zeta }{L^\Phi}
\notag
\\[1ex]
&\le
\NM {\int _{\rr m}U_2(\cdo ,\zeta )\, d\zeta}{L^\Phi},
\label{Eq:KernelU2Est}
\intertext{where}
U_2(X)
&=
U_2(x,y,\zeta )
=
\nm {U_1(X,\cdo -(\fy '_x(X),\fy '_y(X)),-\fy '_\zeta (X) )}{L^\Phi}
\notag
\\[1ex]
&=
\nm {U_1(X,\cdo ,-\fy '_\zeta (X) )}{L^\Phi},
\label{Eq:U2RelU_1}
\intertext{with}
U_1(X,\xi ,\eta ,z)
&=
(G*U(X,\cdo ))(\xi ,\eta ,z).
\label{Eq:U1Def}
\end{align}
By using Minkowski's inequality again, we obtain
\begin{align*}
\int _{\rr m}U_2(X)\, d\zeta
&\le
\nm G{L^1}\int _{\rr m}\nm {U(X,\cdo ,-\fy '_\zeta (X) )}{L^\Phi}\, d\zeta
\\[1ex]
&\le
\nm G{L^1}\int _{\rr m}\sup _{z\in \rr m}\left (
\nm {U(X,\cdo ,z)}{L^\Phi}
\right )
\, d\zeta .
\end{align*}

\par

A combination of these estimates now gives \eqref{Eq:KernelEst},
and thereby (1). The assertion (2) now follows from (1) and Proposition
\ref{Prop:KernelOrlModCont}, and we have proved the result when (i)
is fulfilled.

\par

Next suppose that (ii) holds. Then for $U_2$ in \eqref{Eq:U2RelU_1}
and \eqref{Eq:U1Def} we have
\begin{align*}
\int _{\rr m}U_2(x,y,\zeta )\, d\zeta
&=
\int _{\rr m}
\nm {U_1(x,y,\zeta ,\cdo ,-\fy '_\zeta (X) )}{L^\Phi}
\, d\zeta
\\[1ex]
&\le
\frac 1{\dbar}
\int _{\rr m}
\nm {U_1(x,y,\zeta (x,y,z),\cdo ,z)}{L^\Phi} \, dz
\\[1ex]
&\le
\frac 1{\dbar}
\int _{\rr m}
\sup _{\zeta \in \rr m}\left (
\nm {U_1(x,y,\zeta ,\cdo ,z)}{L^\Phi} \right )\, dz .
\end{align*}
Here in, the first inequality we have taken $z=-\fy '_\zeta (X) $
as new variable of integration, and used that
$|\det (\fy ''_{\zeta ,\zeta})| \ge \dbar$.

\par

A combination of these estimates Proposition
\ref{Prop:KernelOrlModCont} both (1) and (2) in the case when (ii) holds,
giving the result.
\end{proof}

\par



\par

\begin{rem}
Let $s>1$, $d=d_2+d_1$, $\Phi$ be a Young
function such that $1<q_\Phi \le
p_\Phi <\infty$, $q_\Phi =p_\Phi =1$ or $q_\Phi =p_\Phi =\infty$, let
$\phi \in \maclD _{0,s}(\rr {d+m})\setminus 0$,
$\omega$, $\omega _j$, $j=0,1,2$, $v$ and 
$\fy$ be as in \eqref{Eq:WeightsIneq} and 
\eqref{Eq:PhaseFuncCond}.
Also let $a\in \Sigma _s'(\rr {d+m})$ be such that 
$\nmm a$
in \eqref{Eq:NormSymb3} is finite,
and that $|\det (\fy ''_{\zeta ,\zeta})|
\ge \dbar$. Then the following is true:
\begin{enumerate}
\item the kernel $K_{a,\fy}$ of $\op _\fy (a)$ 
belongs to
$M^\Phi _{(\omega _0)}(\rr d)$, and 
\eqref{Eq:KernelEst} holds
for some constant $C$ which is independent of $a\in 
\mascS '(\rr {d+m})$ and
$\fy \in C (\rr {d+m})$;

\vrum

\item the definition of $\op _\fy (a)$ extends uniquely to a continuous operator from
$M^{\Phi ^*}_{(\omega _1)}(\rr {d_1})$ to $M^{\Phi}_{(\omega _2)}(\rr {d_2})$, and
\eqref{Eq:OpEst} holds.
\end{enumerate}
\end{rem}

\par

\section{Schatten-von Neumann
properties for Fourier
integral operators}\label{sec3} 

\par

In this section we discuss Orlicz Schatten-von Neumann
properties for Fourier integral operators with amplitudes,
essentially only depending on two variables,
belonging to Orlicz modulation spaces. In the first part we
consider Fourier integral operators of the form
\begin{align}
\op _\fy (a)f(x)
&=
(2\pi )^{-d}
\iint _{\rr {2d}}a(x,\zeta
)f(y)e^{i\fy (x,y,\zeta )}\, dyd\zeta .
\label{Eq:FIntOp2Again}
\intertext{Thereafter we perform some
extensions to Fourier integral operators of
the form}
\op _{A,\fy} (a)f(x)
&=
(2\pi )^{-d}
\iint _{\rr {2d}}a(x-A(x-y),\zeta
)f(y)e^{i\fy (x,y,\zeta )}\, dyd\zeta .
\tag*{(\ref{Eq:FIntOp2Again})$'$}
\label{Eq:FIntOp3Again}
\intertext{Here $A$ can be any real $d\times d$ matrix. (See also
\eqref{Eq:FIntOp2} and \eqref{Eq:FIntOp3} from the introduction.)
\linebreak
\indent
In the passage from the operators in \eqref{Eq:FIntOp2Again}
to \ref{Eq:FIntOp3Again} it is sometimes convenient
to also let the phase function $\fy$ in \ref{Eq:FIntOp3Again}
depend on $A$ as}
\op _{A,\fy} (a)f(x)
&=
(2\pi )^{-d}
\iint _{\rr {2d}}a(x-A(x-y),\zeta
)f(y)e^{i\fy _A(x,y,\zeta )}\, dyd\zeta ,
\tag*{(\ref{Eq:FIntOp2Again})$''$}
\label{Eq:FIntOp3AgainB}
\intertext{with}
\fy _A(x,y,\zeta )
&=
\fy (x-A(x-y),y-A(x-y),\zeta ).
\label{Eq:fyAFaseFunc}
\end{align}

%

\par

We observe that the map which takes amplitudes to the kernels of
these operators are formally given by
\begin{align}
a &\mapsto K_{a,\fy}(x,y)
\equiv
(2\pi )^{-d}
\int _{\rr {d}}a(x,\zeta
)e^{i\fy (x,y,\zeta )}\, d\zeta ,
\label{Eq:FIntOp2Kernel}
\\[1ex]
a &\mapsto K_{a,\fy ,A}(x,y)
\equiv
(2\pi )^{-d}
\int _{\rr {d}}a(x-A(x-y),\zeta
)e^{i\fy (x,y,\zeta )}\, d\zeta ,
\tag*{(\ref{Eq:FIntOp2Kernel})$'$}
\label{Eq:FIntOp3Kernel}
\intertext{and}
a &\mapsto K_{a,\fy ,A}(x,y)
\equiv
(2\pi )^{-d}
\int _{\rr {d}}a(x-A(x-y),\zeta
)e^{i\fy _A(x,y,\zeta )}\, d\zeta ,
\tag*{(\ref{Eq:FIntOp2Kernel})$''$}
\label{Eq:FIntOp4Kernel}
\end{align}
provided we may interpret the integrals in some sense.

\par

Our first result on this is the following. Here we assume that
the involved weight and phase functions satisfies
%
%
\begin{align}
\omega _0(x,y,\xi +\fy '_x(X),\fy '_y(X))
&\le
C\omega (x,\zeta ,\xi ,-\fy '_\zeta (X)),
\label{Eq:SchattWeightCond0A}
\\[1ex]
\frac {\omega _2(x,\xi )}{\omega _1(y,\eta )}
&\le
C\omega _0(x,y,\xi ,-\eta )
\label{Eq:SchattWeightCond1A}
\\[1ex]
\omega _0(x,y,\xi ,\eta _1+\eta _2)
&\le
C\omega _0(x,y,\xi ,\eta _1)v_1(\eta _2)
\label{Eq:SchattWeightCond2A}
\\[1ex]
\omega (x,\zeta ,\xi _1+\xi _2,z_1+z_2)
&\le
\omega (x,\zeta ,\xi _1,z_1)v_2(\xi _2,z_2),
\label{Eq:SchattWeightCond3A}
\\[1ex]
v(X,\xi ,\eta ,z)
&=
v_1(\eta )v_2(\xi ,z),
\label{Eq:SchattWeightCond4A}
\\[1ex]
x,y,z,z_j,\xi ,\xi _j,\eta ,\eta _j,\zeta &\in \rr d,
\quad X= (x,y,\zeta ),
\quad  j=1,2.
\notag
\end{align}

\par

\begin{prop}
\label{Prop:OrliczKernelsFIO}
Suppose $\Phi$ is a Young function which satisfies a local $\Delta _2$ condition,
$s>1$, $\dbar >0$, $\omega ,\omega _0\in \mascP _s(\rr {4d})$ and
$v\in \mascP _s(\rr {6d})$
satisfy \eqref{Eq:SchattWeightCond0A},
\eqref{Eq:SchattWeightCond2A}--\eqref{Eq:SchattWeightCond4A},
and let $\fy \in C^2(\rr {3d})$ be such that $\fy '' \in M^{\infty ,1}_{(v)}(\rr {3d})$
and
\begin{equation}
\label{Eq:DetphiCond3}
|\det (\fy ''_{y,\zeta })|\ge \dbar .
\end{equation}
Then the map in \eqref{Eq:FIntOp2Kernel}
from $\Sigma _1(\rr{2d})$ to $\Sigma _1'(\rr {2d})$ extends uniquely to
a continuous map from $M^\Phi _{(\omega )}(\rr {2d})$ to
$M^\Phi _{(\omega _0)}(\rr {2d})$, and
\begin{equation}
\label{Eq:KerSymbOrlEst}
\nm {K_{a,\fy}}{M^\Phi _{(\omega _0)}}
\le
C\dbar ^{-1}\exp (\nm \fy{M^{\infty ,1}_{(v)}})\nm a{M^\Phi _{(\omega )}},
\qquad
a\in M^\Phi _{(\omega )}(\rr {2d}),
\end{equation}
for some constant $C$ which is independent of $a$, $\fy$ and $\Phi$.
\end{prop}

\par

For the proof we observe that $M^{\infty ,1}_{(v)}(\rr {3d})
\subseteq C(\rr {3d})\cap L^\infty (\rr {3d})$. Hence the conditions on
$\fy$ in Proposition \ref{Prop:OrliczKernelsFIO} imply that
\begin{equation}
\label{Eq:SpecCondFy}
\dbar \le |\det (\fy ''_{y,\zeta }(x,y,\zeta ))|\le C,
\end{equation}
for some constant $C>0$ which is independent of $x,y,\zeta \in \rr d$.

\par

\begin{proof}[Proof of Proposition \ref{Prop:OrliczKernelsFIO}]
Let $\phi$ and $\phi _j$ be the same as in \eqref{Eq:CouplWindFunc},
where we additionally assume that $\phi$ is given by
$$
\phi (x,y,\zeta ) = \psi _1(x,\zeta )\psi _2(y),
\qquad
x,y,\zeta \in \rr d.
$$
Also let $\psi$ be as in Remark \ref{Rem:LemmaThmCont2}, and let
$$
H=|V_{\psi _1}a \cdot \omega |.
$$
Then $\nm a{M^\Phi _{(\omega )}}=\nm H{L^\Phi}$.

\par

If $\Xi _1=(\xi _1,\eta _1,z_1)\in \rr {3d}$, then
Lemma \ref{Lem:LemmaThmCont2} and
Remark \ref{Rem:LemmaThmCont2} give
\begin{equation}
\label{Eq:STFTKernEst}
\begin{aligned}
&|V_\psi K_{a,\fy}(x,y,\xi ,\eta )\cdot \omega _0(x,y,\xi ,\eta )|
\asymp
\left |
\int _{\rr d} \maclK _{a,\fy}(X,\xi ,\eta )\, d\zeta 
\right |
\\[1ex]
&\le
\iint \limits _{\rr {4d}}H(x,\zeta ,\xi -\fy '_x-\xi _1,-\fy '_\zeta -z_1)
|\widehat \psi _2(\eta -\fy '_y-\eta _1)| G(\Xi _1)\, d\Xi _1d\zeta .
\end{aligned}
\end{equation}
Here observe that
$$
(\fy '_x,\fy '_y,\fy '_\zeta )= (\fy '_x(X),\fy '_y(X),\fy '_\zeta (X) )
$$
depends on $X=(x,y,\zeta )\in \rr {3d}$. In order to deduce the
result we shall consider two cases. In the first case we assume
that $\Phi$ is positive, and in the second case we assume that
$\Phi$ fails to be positive.

\par

Therefore, suppose that $\Phi$ additionally is positive. Then by
replacing $\Phi$ with another Young function which agree with
$\Phi$ near origin, we may assume that $\Phi$ satisfies a global
$\Delta _2$-condition.

\par

By taking $(\Xi _1,\fy '_y)$ as new variables of integrations, 
and using \eqref{Eq:SpecCondFy} and
\eqref{Eq:STFTKernEst}, we obtain
\begin{equation*}
\begin{aligned}
&|V_\psi K_{a,\fy}(x,y,\xi ,\eta )\cdot \omega _0(x,y,\xi ,\eta )|
\\[1ex]
&\le
\frac 1{\dbar}
\iint \limits _{\rr {4d}}H(x,\zeta ,\xi -\fy '_x-\xi _1,-\fy '_\zeta -z_1)
|\widehat \psi _2(\eta -\zeta _1-\eta _1)| G(\Xi _1)\, d\Xi _1d\zeta _1.
\end{aligned}
\end{equation*}
Let $C_0=\nm G{L^1}\nm {\widehat \psi _2}{L^1}$.
Then Jensen's inequality gives
\begin{equation*}
\begin{aligned}
&\Phi ( |V_\psi K_{a,\fy}(x,y,\xi ,\eta )\cdot \omega _0(x,y,\xi ,\eta )| )
\\[1ex]
&\le
\Phi \left (
\frac 1{\dbar}
\iint \limits _{\rr {4d}}H(x,\zeta ,\xi -\fy '_x-\xi _1,-\fy '_\zeta -z_1)
|\widehat \psi _2(\eta -\zeta _1-\eta _1)| G(\Xi _1)\, d\Xi _1d\zeta _1
\right )
\\[1ex]
&\le
\frac 1{C_0}
\iint \limits _{\rr {4d}}
\Phi \left (
\frac {C_0}{\dbar}
H(x,\zeta ,\xi -\fy '_x-\xi _1,-\fy '_\zeta -z_1) \right )
|\widehat \psi _2(\eta -\zeta _1-\eta _1)| G(\Xi _1)\, d\Xi _1d\zeta _1
\\[1ex]
&\le
\frac C{C_0}
\iint \limits _{\rr {4d}}
\Phi \left (
\frac {C_0}{\dbar}
H(x,\zeta ,\xi -\fy '_x-\xi _1,-\fy '_\zeta -z_1) \right )
|\widehat \psi _2(\eta -\fy '_y -\eta _1)| G(\Xi _1)\, d\Xi _1d\zeta .
\end{aligned}
\end{equation*}
In the last step we have taken back the original variables of integration,
and used the second inequality in \eqref{Eq:SpecCondFy}. Since
$\Phi$ satisfies a $\Delta _2$ condition we get
\begin{equation*}
\begin{aligned}
&\Phi ( |V_\psi K_{a,\fy}(x,y,\xi ,\eta )\cdot \omega _0(x,y,\xi ,\eta )| )
\\[1ex]
&\le
C
\iint \limits _{\rr {4d}}
\Phi \left (
H(x,\zeta ,\xi -\fy '_x-\xi _1,-\fy '_\zeta -z_1) \right )
|\widehat \psi _2(\eta -\fy '_y -\eta _1)| G(\Xi _1)\, d\Xi _1d\zeta ,
\end{aligned}
\end{equation*}
for some constant $C>0$.

\par

By integration we get
\begin{equation*}
\begin{aligned}
&
\iiiint \limits _{\rr {4d}}
\Phi ( |V_\psi K_{a,\fy}(x,y,\xi ,\eta )\cdot \omega _0(x,y,\xi ,\eta )| )\, dxdyd\xi d\eta 
\\[1ex]
&\le
C
\iiiint \limits _{\rr {8d}}
\Phi \left (
H(x,\zeta ,\xi -\fy '_x-\xi _1,-\fy '_\zeta -z_1) \right )
|\widehat \psi _2(\eta -\fy '_y -\eta _1)| G(\Xi _1)\, dX d\xi d\eta  d\Xi _1.
\end{aligned}
\end{equation*}
Here recall that $X=(x,y,\zeta )$, giving that the integration variables in the
last integral are given by
$$
(x,y,\zeta ,\xi ,\eta ,\xi _1,\eta _1,z_1)\in \rr {8d}.
$$

\par

By taking
$$
(x,\fy '_\zeta ,\zeta ,\xi ,\eta ,\xi _1,\eta _1,z_1)
$$
as new variables of integrations, and using the first inequality in
\eqref{Eq:SpecCondFy}, we obtain
\begin{equation*}
\begin{aligned}
&
\iiiint \limits _{\rr {4d}}
\Phi ( |V_\psi K_{a,\fy}(x,y,\xi ,\eta )\cdot \omega _0(x,y,\xi ,\eta )| )\, dxdyd\xi d\eta 
\\[1ex]
&\le
\frac C{\dbar}
\iiiint \limits _{\rr {8d}}
\Phi \left (
H(x,\zeta ,\xi -\fy '_x-\xi _1,-y -z_1) \right )
|\widehat \psi _2(\eta -\fy '_y -\eta _1)| G(\Xi _1)\, dX d\xi d\eta  d\Xi _1
\\[1ex]
&=
\frac C{\dbar}
\iiiint \limits _{\rr {8d}}
\Phi \left (
H(x,\zeta ,\xi ,y) \right )
|\widehat \psi _2(\eta )| G(\Xi _1)\, dX d\xi d\eta  d\Xi _1
\\[1ex]
&=
\frac C{\dbar}
\nm G{L^1}\nm {\psi _2}{L^1}\nm {\Phi (H)}{L^1}.
\end{aligned}
\end{equation*}
Here in the first equality we have performed a straight-forward
substitution of integration variables.

\par

The estimate \eqref{Eq:KerSymbOrlEst} now follows from the
last estimates, homogeneity and the fact that $\Phi$ satisfies
a $\Delta _2$ condition. This gives the result when $\Phi$ is positive.

\par

Next suppose that $\Phi$ is not positive. Then $\Phi$ equals to zero near origin,
which implies that $M^\Phi _{(\omega )}=M^\infty _{(\omega )}$ and
$M^\Phi _{(\omega _0)}=M^\infty _{(\omega _0)}$. By
\eqref{Eq:KerSymbOrlEst} we obtain
\begin{equation*}
\begin{aligned}
\nm {K_{a,\fy}}{M^\infty _{(\omega _0)}}
&=
\sup _{x,y,\xi ,\eta}\big (|V_\psi K_{a,\fy}(x,y,\xi ,\eta )\cdot \omega _0(x,y,\xi ,\eta )| \big )
\\[1ex]
&\lesssim
\nm H{L^\infty}
\sup _{x,y,\xi ,\eta} \left (
\iint \limits _{\rr {4d}}
|\widehat \psi _2(\eta -\fy '_y-\eta _1)| G(\Xi _1)\, d\Xi _1d\zeta \right )
\\[1ex]
&\le
\nm H{L^\infty}
\sup _{x,y,\xi ,\eta} \left (
\frac 1{\dbar}
\iint \limits _{\rr {4d}}
|\widehat \psi _2(\eta -\zeta _1-\eta _1)| G(\Xi _1)\, d\Xi _1d\zeta _1 \right )
\\[1ex]
&\le
\nm H{L^\infty}
\asymp
\nm a{M^\infty _{(\omega )}}.
\end{aligned}
\end{equation*}
In the second inequality we have again taking $(\Xi _1,\fy '_y)$ as new
variables of integration. The result now follows in this case from the latter
estimates.
\end{proof}

\par

We shall combine the previous proposition with the following one.

\par

\begin{prop}
\label{Prop:KernelOrlSchatt}
Let $\Phi$ be a quasi-Young function which satisfies
$$
q_\Phi \le p_\Phi < 2
\quad \text{or}\quad
q_\Phi =p_\Phi =2.
$$
Also let
$\omega _0\in \mascP _E(\rr {2d_2+2d_1})$ and $\omega _j\in \mascP _E(\rr {2d_j})$,
$j=1,2$, be such that \eqref{Eq:SchattWeightCond1A} holds. If
$K\in M^\Phi _{(\omega )}(\rr {d_2+d_1})$, then
$T_K\in \mascI _\Phi (\omega _1,\omega _2)$, and
$$
\nm {T_K}{\mascI _\Phi (\omega _1,\omega _2)}
\le
C\nm K{M^\Phi _{(\omega _0)}},
\qquad
K\in M^\Phi _{(\omega )}(\rr {d_2+d_1}),
$$
for some constant $C>0$ which is independent of $K$ and $\Phi$.
\end{prop}

\par

\begin{proof}
Let $d=d_2+d_1$, $\phi ,\psi \in \Sigma _1(\rr {d})$ and $\ep >0$ be chosen such that
\begin{alignat*}{2}
&\{ \phi _{j,\iota} \} _{j,\iota \in \ep \zz d} &
\qquad \text{and}\qquad
&\{ \psi _{j,\iota} \} _{j,\iota \in \ep \zz d} ,
\\[1ex]
&\phi _{j,\iota} = e^{i\scal \cdo \iota}\phi (\cdo -j), &
\qquad
&\psi _{j,\iota} = e^{i\scal \cdo \iota}\psi (\cdo -j) ,
\end{alignat*}
are dual Gabor frames. The existence of such frames follows from e.{\,}g.
\cite{GroLyu}.

\par

For any sequence $c=\{ c(j,\iota )\} _{j,\iota \in \ep \zz {2d}}$, let $T_c$ be
the kernel operator
\begin{equation}
\label{Eq:KernelSynthesis}
T_c =T_K = \sum _{j,\iota} c(j,\iota )T_{\phi _{j,\iota}}
\quad \text{when}\quad
K=K_c=\sum _{j,\iota} c(j,\iota )\phi _{j,\iota}
\end{equation}
(i.{\,}e. $K_c$ is the synthesis of $c$).
Also let $\ON (\omega _j)$ denote  the set of all orthonormal sequences in
$M^2_{(\omega _j)}(\rr {d_j})$, $\{ f_{j,k}\} _{k=1}^\infty \in \ON (\omega _j)$,
$j=1,2$, and let $S$ be the operator from $\ell _0(\ep \zz {2d})$ to
$\ell _0'(\mathbf Z_+)$, given by
$$
S(\{ c(j,\iota )\} _{j,\iota \in \ep \zz {2d}})
=
\{ (T_cf_{1,k},f_{2,k})_{M^2_{(\omega _2)}} \} _{k=1}^\infty .
$$
We have
\begin{align}
\nm {\{ c(j,\iota )\} _{j,\iota \in \ep \zz {2d}}}{\ell ^\Phi _{(\omega )}}
&\asymp
\nm {K_c}{M^\Phi _{(\omega )}}
\label{Eq:GabModNormEquiv}
\intertext{when $c\in \ell _0(\ep \zz {2d})$, and}
\nm {\{ (T_cf_{1,k},f_{2,k})_{M^2_{(\omega _2)}} \} _{k=1}^\infty}{\ell ^\Psi}
&\le
\nm {T_c}{\mascI _\Psi (\omega _1,\omega _2)}.
\label{Eq:SchattNormEst}
\end{align}

\par

We have
$$
\nm {T_K}{\mascI _p(\omega _1,\omega _2)} \le C\nm K{M^p_{(\omega )}},
$$
for some constant $C>0$ which is independent of $K\in M^p_{(\omega )}(\rr d)$ and
$0<p\le 2$ (see e.{\,}g. \cite{Toft32}). Hence it follows by combining these relations that
\begin{equation}
\label{Eq:SOpCont}
\nm {S(c)}{\ell ^p} \le C\nm c{\ell ^p_{(\omega )}},
\end{equation}
for some constant $C>0$ which is independent of $c\in \ell _0(\ep \zz {2d})$,
$\{ f_{j,k}\} _{k=1}^\infty \in \ON (\omega _j)$, $j=1,2$, and $0<p\le 2$. Since
$\ell _0$ is dense in $\ell ^p$ when $0<p\le 2$, it follows that $S$ extends
uniquely to a continuous map from $\ell ^p_{(\omega )}(\ep \zz {2d})$ to
$\ell ^p(\mathbf Z_+)$, and that \eqref{Eq:SOpCont} holds for any
$c\in \ell ^p_{(\omega )}(\ep \zz {2d})$.

\par

By Marcinkiewicz interpolation theorem for Orlicz spaces, given in
\cite{Liu}, it follows that $S$ restricts
to a continuous map from $\ell ^\Phi _{(\omega )}(\ep \zz {2d})$ to
$\ell ^\Phi (\mathbf Z_+)$, and that 
\begin{equation}
\tag*{(\ref{Eq:SOpCont})$'$}
\label{Eq:SOpContB}
\nm {S(c)}{\ell ^\Phi} \le C\nm c{\ell ^\Phi_{(\omega )}},
\qquad
c\in \ell ^\Phi _{(\omega )}(\ep \zz {2d}).
\end{equation}

\par

A combination of \eqref{Eq:GabModNormEquiv},
\eqref{Eq:SchattNormEst} and \ref{Eq:SOpContB} gives
$$
\nm {\{ (T_cf_{1,k},f_{2,k})_{M^2_{(\omega _2)}} \} _{k=1}^\infty}{\ell ^\Psi}
\le
C\nm {K_c}{M^p_{(\omega )}},
\qquad
c\in \ell ^\Phi _{(\omega )}(\ep \zz {2d}).
$$
Since any $K\in M^\Phi _{(\omega )}(\rr {d})$ is given by $K_c$, for some
$c\in \ell ^\Phi _{(\omega )}(\ep \zz {2d})$, the last estimate gives
$$
\nm {\{ (T_Kf_{1,k},f_{2,k})_{M^2_{(\omega _2)}} \} _{k=1}^\infty}{\ell ^\Psi}
\le
C\nm {K_c}{M^p_{(\omega )}},
\qquad
K\in M^\Phi _{(\omega )}(\rr {d}),
$$
where we recall that the constant $C$ is independent of the choices of
$\{ f_{j,k}\} _{k=1}^\infty \in \ON (\omega _j)$, $j=1,2$. By taking the supremum
over all such orthonormal sequences we obtain
$$
\nm {T_K}{\mascI _\Phi (\omega _1,\omega _2)}
\le
C\nm {K_c}{M^p_{(\omega )}},
\qquad
K\in M^\Phi _{(\omega )}(\rr {d}),
$$
which gives the result.
\end{proof}

\par

A straight-forward combination of Propositions \ref{Prop:OrliczKernelsFIO}
and \ref{Prop:KernelOrlSchatt} gives the following. The details are left for the reader.

\par

\begin{thm}
\label{Thm:OrliczSchattenFIO}
Suppose $\Phi$ is a Young function which satisfies a local $\Delta _2$ condition,
either $1< q_\Phi \le p_\Phi <2$ or $1\le q_\Phi = p_\Phi \le 2$,
$s>1$, $\dbar >0$, $\omega ,\omega _0\in \mascP _s(\rr {4d})$ and
$v\in \mascP _s(\rr {6d})$
satisfy \eqref{Eq:SchattWeightCond0A},
\eqref{Eq:SchattWeightCond1A}--\eqref{Eq:SchattWeightCond4A},
$\fy \in C^2(\rr {3d})$ satisfies $\fy '' \in M^{\infty ,1}_{(v)}(\rr {3d})$ and
\eqref{Eq:DetphiCond3}, and let
$a\in M^\Phi _{(\omega )}(\rr {2d})$. Then
$\op _\fy (a)\in \mascI _\Phi (\omega _1,\omega _2)$,
and
\begin{equation}
\label{Eq:FIOOrlSchattEst}
\begin{aligned}
\nm {\op _\fy (a)}{\mascI _\Phi (\omega _1,\omega _2)}
&\le
C\dbar ^{-1}\exp (\nm \fy{M^{\infty ,1}_{(v)}})\nm a{M^\Phi _{(\omega )}},
\\
a&\in M^\Phi _{(\omega )}(\rr {2d}),
\end{aligned}
\end{equation}
for some constant $C$ which is independent of $a$, $\fy$ and $\Phi$.
\end{thm}

\par

Next we extend the previous results and investigations to include
Fourier integral operators of the form \ref{Eq:FIntOp3Again}. The following
proposition is important for this transition. Here for any weight function
$\omega _0$ on $\rr {2d}$ and real $d\times d$ matrix $A$, we let
\begin{align}
&\omega _A(x,y,\xi ,\eta )
\notag
\\[1ex]
&\equiv
\omega _0(x-A(x-y),y-A(x-y),\xi + A^*(\xi +\eta ),\eta - A^*(\xi +\eta )).
\label{Eq:WeightFuncMatrixA}
\intertext{It follows that}
&\omega _0(x,y,\xi ,\eta )
\notag
\\[1ex]
&=
\omega _A(x+A(x-y),y+A(x-y),\xi - A^*(\xi +\eta ),\eta + A^*(\xi +\eta )).
\notag
\end{align}

\par

Here also recall that for a linear bijective map $T$ on $\rr d$, the pullback $T^*f$
of $T$ on a distribution $f$ on $\rr d$ is defined by the formula
$$
\scal {T^*f}\phi = |\det (T)|^{-1}\scal f{\phi (T^{-1}\cdo )},
$$
for any test function $\phi$ in corresponding test function space.

\par

\begin{prop}
\label{Prop:STFTNormEstPullback}
Let $A$ be a real $d\times d$ matrix, $\Phi$ be a Young function, and let
$\omega _0,\omega _A\in \mascP _E(\rr {2d})$
be such that \eqref{Eq:WeightFuncMatrixA} holds. Also let
$T_A: \rr {2d}\to \rr {2d}$ be given by
$$
T_A(x,y)\equiv (x-A(x-y),y-A(x-y)).
$$
Then $T_A^*$ on $\Sigma _1(\rr {2d})$ extends uniquely to a
homoeomorphism from $M^\Phi _{(\omega _0)}(\rr {2d})$ to
$M^\Phi _{(\omega _A)}(\rr {2d})$.
\end{prop}

\par

\begin{proof}
Let $K\in \Sigma _1'(\rr {2d})$, $\phi \in \Sigma _1(\rr {2d})\setminus 0$,
$K_A=T_A^*K$ and $\phi _A=T_A^*\phi $. Then
$\phi _A\in \Sigma _1(\rr {2d})\setminus 0$. By straight-forward computations
it follows that
$$
V_{\phi _A}K_A(x,y,\xi ,\eta )
=
V_\phi K(x-A(x-y),x-A(x-y),\xi + A^*(\xi +\eta ),\eta - A^*(\xi +\eta )).
$$
By multiplying the equality by $\omega _A$
in \eqref{Eq:WeightFuncMatrixA} and applying the $L^\Phi$ norm
we get
\begin{equation}
\label{Eq:STFTNormEstPullback}
\nm {V_{\phi _A}K_A\cdot \omega _A}{L^\Phi}
=
\nm {V_{\phi _A}K\cdot \omega _0}{L^\Phi}.
\end{equation}
In the last equality we have used the fact that
$$
\det 
\left (
\begin{matrix}
I-A & A\\
-A & I+A
\end{matrix}
\right )
=
\det 
\left (
\begin{matrix}
I+A^* & A^*\\
-A^* & I-A^*
\end{matrix}
\right )
=1.
$$
The assertion is now a straight-forward consequence of
\eqref{Eq:STFTNormEstPullback}.
\end{proof}

\par

\begin{rem}
In \cite{ToPfTe}, a broad family of modulation spaces are presented, where
each modulation space $M(\omega ,\mascB )$ is parameterized by a weight
$\omega \in \mascP _E(\rr {2d})$ and normal invariant quasi-Banach function
space $\mascB$ on $\rr {2d}$. If $\omega _0 \in \mascP _E(\rr {4d})$
and $\mascB$ hosts functions defined on $\rr {4d}$, then the previous proof
shows that $T^*A$ in Proposition \ref{Prop:STFTNormEstPullback} is
a homeomorphism from $M(\omega _0,\mascB )$
to $M(\omega _A,\mascB )$. This extends Proposition
\ref{Prop:STFTNormEstPullback} to more general modulation spaces.
\end{rem}

Next we extend Proposition \ref{Prop:OrliczKernelsFIO} to
Fourier integral operators of the form \ref{Eq:FIntOp3Again}.
Here the conditions of involved weight functions need to
be modified into
\begin{multline}
\tag*{(\ref{Eq:SchattWeightCond0A})$'$}
\label{Eq:SchattWeightCond0AB}
\omega _0(x+A(x-y),y+A(x-y),(I-A^*)\xi +\fy '_x(X),A^*\xi + \fy '_y(X))
\\[1ex]
\le
C\omega (x,\zeta ,\xi ,-\fy '_\zeta (X)),
\quad X= (x,y,\zeta ).
\end{multline}

\par

\renewcommand{\rubrik}{Proposition \ref{Prop:OrliczKernelsFIO}$'$\!}

\par

\begin{tom}
Suppose $\Phi$ is a Young function which satisfies a local $\Delta _2$ condition,
$s>1$, $\dbar >0$, $\omega ,\omega _0\in \mascP _s(\rr {4d})$ and
$v\in \mascP _s(\rr {6d})$
satisfy {\rm{\ref{Eq:SchattWeightCond0AB}}},
\eqref{Eq:SchattWeightCond2A}--\eqref{Eq:SchattWeightCond4A},
and let $\fy \in C^2(\rr {3d})$ be such that $\fy '' \in M^{\infty ,1}_{(v)}(\rr {3d})$
and
\begin{equation}
\tag*{(\ref{Eq:DetphiCond3})$'$}
\label{Eq:DetphiCond3A}
\left |
\det
\left (
\fy ''_{y,\zeta}  -A^*(\fy ''_{x,\zeta} + \fy ''_{y,\zeta} )
\right )
\right |
\ge
\dbar .
\end{equation}
Then the map in {\rm{\ref{Eq:FIntOp3Kernel}}}
from $\Sigma _1(\rr{2d})$ to $\Sigma _1'(\rr {2d})$ extends uniquely to
a continuous map from $M^\Phi _{(\omega )}(\rr {2d})$ to
$M^\Phi _{(\omega _0)}(\rr {2d})$, and
\begin{equation}
\tag*{(\ref{Eq:KerSymbOrlEst})$'$}
\label{Eq:KerSymbOrlEstA}
\nm {K_{a,\fy ,A}}{M^\Phi _{(\omega _0)}}
\le
C\dbar ^{-1}\exp (\nm \fy{M^{\infty ,1}_{(v)}})\nm a{M^\Phi _{(\omega )}},
\qquad
a\in M^\Phi _{(\omega )}(\rr {2d}),
\end{equation}
for some constant $C$ which is independent of $a$, $\fy$ and $\Phi$.
\end{tom}

\par

\begin{proof}
We shall reduce ourselves to the case when $A=0$, and then
Proposition \ref{Prop:OrliczKernelsFIO} will give the result.

\par

Let
\begin{align*}
\psi (x,y,\zeta )
&\equiv
\fy (x+A(x-y),y+A(x-y),\zeta )
\intertext{and}
\vartheta (x,y,\xi ,\eta )
&\equiv
\omega _0(x+A(x-y),y+A(x-y),\xi -A^*(\xi +\eta ),\eta +A^*(\xi +\eta )).
\intertext{Then}
\fy (x,y,\zeta )
&=
\psi (x-A(x-y),y-A(x-y),\zeta )
\intertext{and}
\omega _0(x,y,\xi ,\eta )
&=
\vartheta (x-A(x-y),y-A(x-y),\xi +A^*(\xi +\eta ),\eta -A^*(\xi +\eta )).
\end{align*}
Furthermore, the change rule gives
\begin{alignat}{1}
\psi '_x &= \fy '_x+A^*(\fy '_x+\fy '_y),
\quad
\psi '_y = \fy '_y+A^*(\fy '_x+\fy '_y),
\quad
\psi ' _\zeta = \fy '_\zeta
\label{Eq:PhaseFuncPullbackTransf1}
\\[1ex]
\fy '_x &= \psi '_x-A^*(\fy '_x+\fy '_y),
\quad
\fy '_y = \psi '_y-A^*(\fy '_x+\fy '_y),
\label{Eq:PhaseFuncPullbackTransf2}
\intertext{and}
\psi _{y,\zeta}'' &= \fy _{y,\zeta}'' -A^*(\fy _{x,\zeta}'' + \fy _{y,\zeta}'').
\label{Eq:PhaseFuncPullbackTransf3}
\end{alignat}

\par

By \ref{Eq:DetphiCond3A} and \eqref{Eq:PhaseFuncPullbackTransf3}
we get $|\det (\psi ''_{y,\zeta })|\ge \dbar$. In particular,
\eqref{Eq:DetphiCond3} holds with $\psi$ in place of $\fy$.

\par

Let $z_1=x+A(x-y)$ and $z_2=y+A(x-y)$. A combination of
\ref{Eq:SchattWeightCond0AB}
and \eqref{Eq:PhaseFuncPullbackTransf1} also gives
\begin{align*}
&\vartheta (x,y,\xi +\psi '_x,\psi '_y)
\\[1ex]
&=
\omega _0(z_1,z_2,\xi +\psi '_x -A^*(\xi +\psi '_x+\psi '_y),
\psi '_y +A^*(\xi +\psi '_x+\psi '_y))
\\[1ex]
&=
\omega _0(z_1,z_2,(I-A^*)\xi +\fy '_x ,A^*\xi +\fy '_y )
\\[1ex]
&=
\omega _0(x+A(x-y),y+A(x-y),(I-A^*)\xi +\fy '_x ,A^*\xi +\fy '_y ).
\\[1ex]
&\le
C\omega (x,\zeta ,\xi ,-\fy '_\zeta )
=
C\omega (x,\zeta ,\xi ,-\psi '_\zeta ),
\end{align*}
which shows that \eqref{Eq:SchattWeightCond0A} holds with $\psi$
in place of $\fy$.

\par

Consequently, all assumptions in Proposition \ref{Prop:OrliczKernelsFIO}
are fulfilled with $\psi$ and $\vartheta$ in place of $\fy$ and $\omega _0$,
respectively. Hence Proposition \ref{Prop:OrliczKernelsFIO} 
and in particular \eqref{Eq:KerSymbOrlEst} give $K_{a,\psi}\in M^\Phi _{(\vartheta )}$,
and that
\begin{equation}
\tag*{(\ref{Eq:KerSymbOrlEst})$''$}
\label{Eq:KerSymbOrlEstB}
\nm {K_{a,\psi}}{M^\Phi _{(\vartheta )}}
\le
C\dbar ^{-1}\exp (\nm \fy{M^{\infty ,1}_{(v)}})\nm a{M^\Phi _{(\omega )}},
\qquad
a\in M^\Phi _{(\omega )}(\rr {2d}).
\end{equation}

\par

Since 
$$
K_{a,\fy ,A}(x,y)
=
K_{a,\psi}(x-A(x-y),y-A(x-y)),
$$
it follows from Proposition \ref{Prop:STFTNormEstPullback} that
$$
\nm {K_{a,\fy ,A}}{M^\Phi _{(\omega _0)}}
\asymp
\nm {K_{a,\psi}}{M^\Phi _{(\vartheta )}}.
$$
The estimate \ref{Eq:KerSymbOrlEstA} now follows by combining
the last relation with \ref{Eq:KerSymbOrlEstB}, and the
result follows.
\end{proof}

\par

In similar ways as for Theorem \ref{Thm:OrliczSchattenFIO} and its
proof, we get the following extension by combining
Proposition \ref{Prop:KernelOrlSchatt} with Proposition
\ref{Prop:OrliczKernelsFIO}$'$. The details are left for the reader.

\par

\renewcommand{\rubrik}{Theorem \ref{Thm:OrliczSchattenFIO}$'$\!}

\par

\begin{tom}
Suppose $A$ is a real $d\times d$ matrix, $\Phi$ is a Young function
which satisfies a local $\Delta _2$ condition,
either $1< q_\Phi \le p_\Phi <2$ or $1\le q_\Phi = p_\Phi \le 2$,
$s>1$, $\dbar >0$, $\omega ,\omega _0\in \mascP _s(\rr {4d})$ and
$v\in \mascP _s(\rr {6d})$
satisfy {\rm{\ref{Eq:SchattWeightCond0AB}}},
\eqref{Eq:SchattWeightCond1A}--\eqref{Eq:SchattWeightCond4A},
$\fy \in C^2(\rr {3d})$ satisfies $\fy '' \in M^{\infty ,1}_{(v)}(\rr {3d})$
and {\rm{\ref{Eq:DetphiCond3A}}}, and let
$a\in M^\Phi _{(\omega )}(\rr {2d})$. Then
$\op _{A,\fy} (a)\in \mascI _\Phi (\omega _1,\omega _2)$,
and
\begin{equation}
\tag*{(\ref{Eq:FIOOrlSchattEst})$'$}
\label{Eq:FIOOrlSchattEstA}
\begin{aligned}
\nm {\op _{A,\fy} (a)}{\mascI _\Phi (\omega _1,\omega _2)}
&\le
C\dbar ^{-1}\exp (\nm \fy{M^{\infty ,1}_{(v)}})\nm a{M^\Phi _{(\omega )}},
\\
a&\in M^\Phi _{(\omega )}(\rr {2d}),
\end{aligned}
\end{equation}
for some constant $C$ which is independent of $a$, $\fy$ and $\Phi$.
\end{tom}

\end{document}